\documentclass[preprint,11pt]{article}
\usepackage{amsmath,amssymb,amsbsy,amsfonts,bbm,mathrsfs,amsthm,amscd}
\usepackage{graphicx,textcomp}
\usepackage{hyperref}

\def\hX{\widehat{X}}
\def\hx{\widehat{x}}
\def\rank{{\rm rank}}
\def\Tr{{\rm Tr}}
\def\tR{\tilde{R}}
\def\reals{{\mathbb R}}
\def\cliq{{\sf cliq}}
\def\hW{\hat{W}}
\def\tX{\tilde{X}}
\def\C{\mathcal{C}}
\def\Ci{\mathcal{C}_i}
\def\Cj{\mathcal{C}_j}
\def\Q{\mathcal{Q}}
\def\L{\mathcal{L}}
\def\bl{\beta^{(\ell)}}
\def\hb{\hat{\beta}}
\def\txl{\tilde{x}^{(\ell)}}
\def\xl{x^{(\ell)}}
\def\E{\mathbb{E}}
\def\P{\mathbb{P}}
\def\cP{\mathcal{P}}
\def\cR{\mathcal{R}}
\def\ve{{\varepsilon}}

\begin{document}
%
%
%
%
%
%
%
%
%
%
%
%
%
%
%
%
%
%
%
%
\title{Localization from Incomplete Noisy Distance Measurements}

\author{Adel Javanmard${}^{*}$  \and Andrea Montanari
\footnote{Department of Electrical Engineering, Stanford University}
\footnote{Department of Statistics, Stanford University}}
\date{}

\newtheorem{thm}{Theorem}[section]

\newtheorem{prepro}{{\bf Proposition}}[section]
\newenvironment{pro}{\begin{prepro}{\hspace{-0.5
               em}{\bf}}}{\end{prepro}}

\newtheorem{precor}{{\bf Corollary}}[section]
\newenvironment{cor}{\begin{precor}{\hspace{-0.5
               em}{\bf}}}{\end{precor}}

\newtheorem{preconj}{{\bf Conjecture}}
\newenvironment{conj}{\begin{preconj}{\hspace{-0.5
               em}{\bf}}}{\end{preconj}}

\newtheorem{preremark}{{\bf Remark}}[section]
\newenvironment{remark}{\begin{preremark}\rm{\hspace{-0.5
               em}{\bf}}}{\end{preremark}}

\newtheorem{predef}{{\bf Definition}}
\newenvironment{define}{\begin{predef}\rm{\hspace{-0.5
               em}{\bf}}}{\end{predef}}

\newtheorem{prelem}{{\bf Lemma}}[section]
\newenvironment{lem}{\begin{prelem}{\hspace{-0.5
               em}{\bf}}}{\end{prelem}}
               
\newtheorem{preclaim}{{\bf Claim}}[section]
\newenvironment{claim}{\begin{preclaim}{\hspace{-0.5
               em}{\bf}}}{\end{preclaim}}

\maketitle

\begin{abstract}
We consider the problem of positioning a cloud of points in the Euclidean space $\mathbb{R}^d$, using noisy measurements of a subset of pairwise distances. This task has applications in various areas, such as sensor network localization and reconstruction of protein conformations from NMR measurements. Also, it is closely related to dimensionality reduction problems and manifold learning, where the goal is to learn the underlying global geometry of a data set using local (or partial) metric information. Here we propose a reconstruction algorithm based on semidefinite programming. For a random geometric graph model  and uniformly bounded noise, we provide a precise characterization of the algorithm's performance: In the noiseless case, we find a radius $r_0$ beyond which the algorithm reconstructs the exact positions (up to rigid transformations). In the presence of noise, we obtain upper and lower bounds on the reconstruction error that match up to a factor that depends only on the dimension $d$, and the average degree of the nodes in the graph.
\end{abstract}

\section{Introduction}
\subsection{Problem Statement}
Given a set of $n$ nodes in $\reals^d$, the  \emph{localization}
problem requires to reconstruct the positions of the nodes from a set
of pairwise measurements $\tilde{d}_{ij}$ for $(i,j)\in E\subseteq
\{1,\dots,n\}\times\{1,\dots,n\}$.
An instance of the problem is therefore given by the graph $G=(V,E)$,
$V=\{1,\dots,n\}$, and the vector of distance measurements
$\tilde{d}_{ij}$ associated to the edges of this graph.

In this paper we consider the random geometric graph model $G(n,r) =
(V, E)$ whereby the  $n$ nodes in $V$ are independent and uniformly
random in the $d$-dimensional hypercube 
$[-0.5,0.5]^d$, and $E \in V \times V$ is a set of edges that connect
the nodes which are close to each other.
More specifically  we let $(i,j) \in E$ if and only if $d_{ij} = \|x_i -x_j\| \leq r$. For each edge $(i,j) \in E$, $\tilde{d}_{ij}$ denotes the measured distance between nodes $i$ and $j$. Letting $z_{ij}\equiv\tilde{d}^2_{ij} -d^2_{ij}$
the measurement error, we will study a \emph{``worst case model''},
in which the errors $\{z_{ij}\}_{(i,j)\in E}$ are arbitrary but uniformly bounded 
$|z_{ij}| \leq \Delta$.
We propose an algorithm for this problem based on semidefinite
programming and provide a rigorous analysis of its 
performance, focusing in particular on its robustness properties. 

Notice that the positions of the nodes can only be determined up to
rigid transformations (a combination of rotation, reflection and
translation) of the nodes, because the inter point distances are
invariant to rigid transformations. For future use, we introduce a formal definition of rigid transformation. Let $X \in \reals^{n \times d}$
be the matrix whose $i^{th}$ row, $x_i^T \in \reals^d$, is the coordinate of node $i$. Further, let ${\sf O}(d)$ denote the orthogonal group of $d \times d$ matrices. A set of positions $Y\in \reals^{n \times d}$ is a rigid transform of $X$, if there exists a $d$-dimensional shift vector $s\in \reals^d$ and an orthogonal matrix $O \in {\sf O}(d)$ such that 
\begin{eqnarray}
Y = XO + u s^T.
\end{eqnarray}
Throughout $u \in \mathbb{R}^n$ is the all-ones vector. Therefore, $Y$ is obtained as a result of first rotating (and/or reflecting) nodes in position $X$ by matrix $O$ and then shifting by $s$. Also, two position matrices $X$ and $Y$ are called equivalent up to rigid transformation, if there exists $O \in {\sf O}(d)$ and a shift $s \in \reals^d$ such that $Y = XO+us^T$. We use the following metric, similar to the one defined in~\cite{Oh:Localization}, to evaluate the distance between the original position matrix $X \in \mathbb{R}^{n \times d}$ and the estimation $\hX \in \mathbb{R}^{n \times d}$. Let $L = I -(1/n)uu^T$ be the centering matrix . Note that $L$ is an $n \times n$ symmetric matrix of rank $n-1$ which eliminates the contribution of the translation, in the sense that $LX = L(X+us^T)$ for all $s\in \reals^d$.  Furthermore, $LXX^TL$ is invariant under rigid transformation and $LXX^L  = L\widehat{X}\widehat{X}^TL$ implies that $X$ and $\widehat{X}$ are equal up to rigid transformation. The metric is defined as 
\begin{equation}
d(X,\hX) \equiv \frac{1}{n^2} \|LXX^TL - L\hX\hX^TL\|_1.
\end{equation}
This is a measure of the average reconstruction error per point, when
$X$ and $\hX$ are aligned optimally.
To get a better intuition about this metric,
consider the case in which all the entries of $LXX^TL - L\hX\hX^TL$ are roughly of the
same order. Then
\begin{eqnarray*}
d(X,\hX) \approx d_2(X,\hX) = \frac{1}{n} \|LXX^TL -
L\hX\hX^TL\|_F\, .
\end{eqnarray*}
Denote by $Y=\hX-X$ the estimation error, and assume without loss of
generality  that both $X$ and $\hX$ are centered.
Then for small $Y$, we have
\begin{eqnarray*}
 d_2(X,\hX) =  \frac{1}{n} \|XY^T +
YX^T+YY^T\|_F   \approx \frac{1}{n} \|XY^T +
YX^T\|_F \\
\stackrel{(a)}{\ge } \frac{C}{\sqrt{n}} \|Y\|_F
 = C\left\{\frac{1}{n}\sum_{i=1}^n \|\hx_i-x_i\|^2\right\}^{1/2}\, ,
\end{eqnarray*}  
where the bound$(a)$ holds with high probability for a
suitable constant $C$, if $X$ is distributed according to our 
model.\footnote{Estimates of this type will be repeatedly proved in the following .}

\smallskip
\noindent \textbf{Remark.} Clearly, connectivity of $G$ is a necessary assumption for the localization problem to be solvable. It is a well known result that  
the graph $G(n,r)$ is connected w.h.p. if $K_d r^d > (\log{n} + c_n)/n
$, where $K_d$ is the volume of the $d$-dimensional unit ball  and
$c_n \to \infty$ \cite{Penrose:RGG}. Vice versa, the graph is
disconnected with positive probability 
 if  $K_d r^d \le (\log{n} + C)/n $ for some constant $C$. Hence, we focus on the regime where $r \geq \alpha (\log{n}/n)^{1/d}$ for some constant $\alpha$. We further notice that, under the random geometric graph
model, the  configuration of the points is almost surely
\emph{generic}, in the sense that the coordinates  do not satisfy any nonzero polynomial equation with integer coefficients.

\subsection{Algorithm and main results}
\label{subsec:algorithm}
The following algorithm uses semidefinite programming (SDP) to solve the localization problem.
\begin{center}
     \begin{tabular}[c]{l}
     \hline
     \textbf{Algorithm} SDP-based Algorithm for Localization\\
     \hline
     \textbf{Input:} dimension $d$, distance measurements $\tilde{d_{ij}} $\\ for $(i,j) \in E$,
     bound on the measurement noise $\Delta$\\
     \textbf{Output:} estimated coordinates in $\mathbb{R}^d$\\
     1:\quad Solve the following SDP problem:\\
   
      \begin{tabular} {l l}
       \qquad minimize &  $\Tr(Q)$\\
       \qquad s.t. & $\left | \langle M_{ij}, Q \rangle - \tilde{d_{ij}}^2 \right | \leq \Delta,$ \quad \quad $(i,j) \in E$\\
       \qquad \quad & $Q \succeq \bf{0}$. \\
      \end{tabular}\\
      
      2:\quad Compute the best rank-$d$ approximation $U_d\Sigma_dU_d^T$ of  $Q$\\
      3:\quad Return $\hX =U_d\Sigma_d^{1/2} $.\\
      \hline
      \end{tabular}
\end{center}  
Here $M_{ij} = e_{ij}e_{ij}^T \in \mathbb{R}^{n \times n}$, where $e_{ij} \in \mathbb{R}^n$ is the vector with $+1$ at $i^{th}$ position, $-1$ at $j^{th}$ position and zero everywhere else. Also, $\langle A, B\rangle \equiv \Tr(A^TB)$. 
Note that with a slight abuse of notation, the solution of the SDP problem in the first step is denoted by $Q$. 

Let $Q_0 := XX^T$ be the Gram matrix of the node positions, namely $Q_{0,ij} = x_i\cdot x_j$. A key observation is that $Q_0$ is a low rank matrix: $\rank(Q_0) \le d$,
and obeys the constraints of the SDP problem. By minimizing $\Tr(Q)$ in the first step, we promote low-rank solutions $Q$ (since  $\Tr(Q)$ is the sum of the eigenvalues of $Q$). Alternatively, this minimization can be interpreted as setting the center of 
gravity of $\{x_1,\dots,x_n\}$ to coincide with the origin, thus removing the 
degeneracy due to translational invariance.

In step 2, the algorithm computes the eigen-decomposition of $Q$ and retains the $d$ largest eigenvalues. This is equivalent to computing the best rank-$d$ approximation of $Q$ in Frobenius norm. The center of gravity of the reconstructed points remains at the origin after this operation.

Our main result provides a characterization of  the robustness 
properties of the SDP-based algorithm. Here and below `with high probability (w.h.p.)'
means with probability converging to $1$ as $n\to\infty$ 
for $d$ fixed. 
\begin{thm}
\label{thm:main_result}
Let $\{x_1,\dots,x_n\}$ be $n$ nodes distributed uniformly at random in the hypercube $[-0.5,0.5]^d$. Further, assume connectivity radius $r \geq \alpha (\log n / n)^{1 /d}$, with $\alpha \geq 10 \sqrt{d}$. Then w.h.p., the error distance between the estimate $\hX$ returned by the SDP-based algorithm and the correct coordinate matrix $X$ is upper bounded as

\begin{equation}
\label{eq:main_result1}
d(X,\hX) \leq C_1 (n r^d)^5 \frac{\Delta}{r^4}.
\end{equation}
Conversely, w.h.p., there exist adversarial
measurement errors $\{z_{ij}\}_{(i,j)\in E}$ such that 
\begin{equation}
\label{eq:main_result2}
d(X,\hX) \ge C_2 \min\{\frac{\Delta}{r^4},1\}\, ,
\end{equation}
Here, $C_1$ and $C_2$ denote universal constants that depend only on $d$.
\end{thm}

The proof of this theorem relies on several technical  results
of independent interest. First, we will prove a general
deterministic error estimate in terms of the condition number of 
the stress matrix of the graph $G$, see Theorem
\ref{thm:Rtilde_bound}.
Next we will use probabilistic arguments to control the stress matrix of random geometric
graphs, see Theorem \ref{thm:OmegaL}. Finally, we will prove several
estimates on the rigidity matrix of $G$, cf. in particular Theorem
\ref{thm:cheeger}. The necessary background in rigidity theory is
summarized in Section \ref{sec:rigidity}.
%
%
\subsection{Related work}

The localization problem and its variants have attracted significant interest over the past years due to their applications in numerous areas, such as sensor network localization~\cite{Biswas:SDP}, NMR spectroscopy~\cite{Lu:protein}, and manifold learning~\cite{Saul:LLE,Tenenbaum:ISOMAP}, to name a few.

Of particular interest to our work are the algorithms proposed for the
localization problem~\cite{Oh:Localization, Singer:positioning,
  Biswas:SDP,Weinberger06anintroduction}. In general, few performance
guarantees have been proved for these algorithms, in particular  in the presence of noise. 

The existing algorithms can be categorized in to two groups. The first group consists of algorithms who try first to estimate the missing distances and then use MDS to find the positions from the reconstructed distance matrix~\cite{Oh:Localization,Cox:MDS}. MDS-MAP~\cite{Cox:MDS} and ISOMAP~\cite{Tenenbaum:ISOMAP} are two well-known examples of this class where the missing entries of the distance matrix are approximated by computing the shortest  paths between all pairs of nodes. The algorithms in the second group formulate the localization problem as a non-convex optimization problem and then use different relaxation schemes to solve it. An example of this type is relaxation to an SDP~\cite{Biswas:SDP, So05theoryof, zhu:universal,Alfakih:SDP,Weinberger06anintroduction}. A crucial assumption in these works is the existence of some anchors among the nodes whose exact positions are known. The SDP is then used to efficiently check whether the graph is uniquely $d$-localizable and to find its unique realization.

Maximum Variance Unfolding (MVU) is an SDP-based algorithm with a very
similar flavor as ours~\cite{Weinberger06anintroduction}. MVU is an
approach to solving dimensionality reduction problems using local
metric information and is based on the following simple
interpretation. Assume $n$ points lying on a low dimensional manifold
in a high dimensional ambient space. In order to find a low
dimensional representation of this data set, the algorithm attempts to
somehow unfold the underlying manifold. To this end, MVU pulls the
points apart in the ambient space, maximizing the total sum of their
pairwise distances, while respecting the local information. However,
to the best of our knowledge, no performance guarantee has been proved
for the MVU algorithm.

Given the large number of applications, and computational methods
developed in this broad area, the present paper is in many respects a
first step. While we focus on a specific model, and a relatively
simple algorithm, we expect that the techniques developed here 
will be applicable to a broader setting, and to a number of algorithms
in the same class.

\subsection{Organization of the paper}
The remainder of this paper is organized as follows. Section~\ref{sec:preliminaries} is a brief review of some notions in rigidity theory and some properties of $G(n,r)$ which will be useful in this paper. In Section~\ref{sec:discussion}, we discuss the implications of Theorem~\ref{thm:main_result} in different applications. The proof of Theorem~\ref{thm:main_result} (upper bound) is given in Section~\ref{sec:main_thm_proof}. Sections~\ref{sec:bound on Rperp} and~\ref{sec:bound on Rtilde} contain the proof of two important lemmas used in proving Theorem~\ref{thm:main_result}. Several technical steps are discussed in Appendices. Finally, We prove Theorem~\ref{thm:main_result} (lower bound) in Section~\ref{sec:converse}.

For the reader's convenience, an overview of the symbols used throughout this paper is given in Table~\ref{tab:notation} in Appendix~\ref{App:tab_notation}.

\section{Preliminaries}
\label{sec:preliminaries}
\subsection{Rigidity Theory}
\label{sec:rigidity}
Rigidity theory studies whether a given partial set of pairwise distances $d_{ij} = \|x_i -x_j\|$ between a finite set of nodes in $\reals^d$ uniquely determine the coordinates of the points up to rigid transformations. This section is a very brief overview of definitions and results in rigidity theory which will be useful in this paper. We refer the interested reader to~\cite{gortler-2010-4,Asimow:rigidity}, for a thorough discussion.

A \emph{framework} $G_X$ in $\reals^d$ is an undirected graph $G =
(V,E)$ along with a \emph{configuration} $X\in\reals^{n\times d}$
which assigns a point $x_i \in \reals^d$ to each vertex $i$ of the
graph. The edges of $G$ correspond to the distance constraints. In the
following, we discuss two important notions, namely Rigidity matrix
and Stress matrix. As 
mentioned above, a crucial part of the proof of
Theorem~\ref{thm:main_result} 
consists in  establishing some properties of the stress matrix and of
the rigidity matrix of the random geometric graph $G(n,r)$.

\smallskip
\noindent \textbf{Rigidity matrix.}
Consider a motion of the framework with $x_i(t)$ being the position vector of point $i$ at time $t$. Any smooth motion that instantaneously preserves the distance $d_{ij}$ must satisfy $\frac{d}{dt} \Arrowvert x_i - x_j \Arrowvert ^2 = 0$ for all edges $(i,j)$. Equivalently,
\begin{equation}
\label{eq:motion_eq}
(x_i - x_j)^T (\dot{x_i} - \dot{x_j}) = 0 \quad \forall (i,j) \in E,
\end{equation} 
where $\dot{x_i}$ is the velocity of the $i^{th}$ point. Given a framework $G_X \in \mathbb{R}^d$, a solution $\dot{X} = \lbrack \dot{x}_1^T \;\dot{x}_2^T \;\cdots \;\dot{x}_n^T \rbrack^T$, with $\dot{x}_i \in \mathbb{R}^d$, for the linear system of equations~\eqref{eq:motion_eq} is called an \emph{infinitesimal motion} of the framework $G_X$. This linear system of equations consists of $|E|$ equations in $dn$ unknowns and can be written in the matrix form $R_G(X) \dot{X} = 0$, where $R_G(X)$ is called the $ |E| \times dn$ \emph{rigidity matrix} of $G_X$.

It is easy to see that for every anti-symmetric matrix $A \in \mathbb{R}^{d \times d}$ and for every vector $b \in \mathbb{R}^d$, $\dot{x}_i = Ax_i + b$ is an infinitesimal motion. Notice that these motions are the derivative of rigid transformations. ($A$ corresponds to orthogonal transformations and $b$ corresponds to translations). Further, these motions span a $d(d+1)/2$ dimensional subspace of $\reals^{dn}$, accounting $d(d-1)/2$ degrees of freedom for orthogonal transformations (corresponding to the choice of $A$), and $d$ degrees of freedom for translations (corresponding to the choice of $b$). Hence, dim $\textrm{Ker}(R_G(X)) \geq d(d+1)/2$. A framework is said to be \emph{infinitesimally rigid} if dim $\textrm{Ker}(R_G(X)) = d(d+1)/2$. 

\smallskip
\noindent \textbf{Stress matrix.}
A \emph{stress} for  a framework $G_X$ is an assignment of scalars $\omega_{ij}$ to the edges such that for each  $i \in V$,
 \begin{equation*}
 \sum_{j: (i,j) \in E}\omega_{ij} (x_i - x_j) = (\sum_{j: (i,j) \in E} \omega_{ij})x_i - \sum_{j: (i,j) \in E} \omega_{ij} x_j = 0.
 \end{equation*}
 A stress vector can be rearranged into an $n \times n$ symmetric matrix $\Omega$ , known as the \emph{stress matrix}, such that for $i \neq j$, the $(i,j)$ entry of $\Omega$ is $\Omega_{ij} = -\omega_{ij}$, and the diagonal entries for $(i,i)$ are $\Omega_{ii} = \sum_{j: j\neq i} \omega_{ij}$. Since all the coordinate vectors of the configuration as well as the all-ones vector are in the null space of $\Omega$, the rank of the stress matrix for generic configurations is at most $n - d - 1$.  
 

There is an important relation between stress matrices of a framework and the notion of \emph{global rigidity}. A framework $G_X$ is said to be \emph{globally rigid} in $\mathbb{R}^d$ if all frameworks in $\mathbb{R}^d$ with the same set of edge lengths are congruent to $G_X$, i.e. are a rigid transformation of $G_X$. Further, a framework  $G_X$ is \emph{generically globally rigid} in $\mathbb{R}^d$ if  $G_X$ is globally rigid at all generic configurations $X$. (Recall that a configuration of points is called \emph{generic} if the coordinates of the points do not satisfy any nonzero polynomial equation with integer coefficients).

The connection between global rigidity and stress matrices is
demonstrated in the following 
two results proved in \cite{connelly:generic} and
\cite{gortler-2010-4}.

 \begin{thm}[Connelly, 2005]
 \label{thm:stress1}
 If $X$ is a generic configuration in $\mathbb{R}^d$ with a stress matrix $\Omega$ of rank $n - d -1$, then $G_X$ is globally rigid in $\mathbb{R}^d$.
 \end{thm}
 
 \begin{thm}[Gortler, Healy, Thurston, 2010]
 \label{thm:stress2}
 Suppose that $X$ is a generic configuration in $\mathbb{R}^d$, such that $G_X$ is globally rigid in $\mathbb{R}^d$. Then either $G_X$ is a simplex or it has a stress matrix $\Omega$ with rank $n- d -1$. 
 \end{thm}
 
 Among other results in this paper, we construct a special stress matrix $\Omega$ for the random geometric graph $G(n,r)$. We also provide upper bound and lower bound on the maximum and the minimum nonzero singular values of this stress matrix. These bounds   are used in proving Theorem~\ref{thm:main_result}.

\subsection{Some Properties of $G(n,r)$}
In this section, we study some of the basic properties of $G(n,r)$ which will be used several times throughout the paper.

Our first remark provides probabilistic bounds on the number of nodes contained in a region $\mathcal{R} \subseteq [-0.5,0.5]^d$.

\begin{remark}{\bf[Sampling Lemma]}
\label{rem:region}
Let $\mathcal{R}$ be a measurable subset  of the hypercube $[-0.5,0.5]^{d}$, and let $V(\mathcal{R})$ denote its volume. Assume $n$ nodes are deployed uniformly at random in $[-0.5,0.5]^d$, and let $n(\mathcal{R})$ be the number of nodes in region $\mathcal{R}$. Then,
\begin{equation}
\label{eqn:region}
n(\mathcal{R}) \in n V(\mathcal{R}) + [-\sqrt{2cn V(\mathcal{R}) \log n},
\sqrt{2cn V(\mathcal{R}) \log n}],
\end{equation}
with probability at least $1-2/n^c$.
\end{remark}  
The proof is immediate and deferred to Appendix~\ref{App:region}.

In the graph $G(n,r)$, every node is connected to all the nodes within its $r$-neighborhood. Using Remark~\ref{rem:region} for $r$-neighborhood of each node, and the fact $r \ge 10\sqrt{d} (\log n / n)^{1/d}$, we obtain the following corollary after applying union bound over all the $r$-neighborhoods of the nodes. 
 
 \begin{cor}
 \label{cor:deg_GNR}
In the graph $G(n,r)$, with $r \ge 10 \sqrt{d} (\log n /n)^{1/d}$, the degrees of all nodes are in the interval $[(1/2)K_dnr^d,(3/2)K_dnr^d]$, with high probability. Here, $K_d$ is the volume of the $d$-dimensional unit ball.
 \end{cor}
 
 Next, we discuss some properties of the spectrum of $G(n,r)$. 
 
Recall that the Laplacian $\mathcal{L}$ of
the graph $G$ is the symmetric matrix indexed by the vertices $V$, such that
$\mathcal{L}_{ij} = -1$ if $(i,j)\in E$, $\mathcal{L}_{ii}$ =
degree$(i)$ and $\mathcal{L}_{ij}=0$ otherwise. The all-ones vector $u \in
\reals^n$ is an eigenvector of $\L(G)$ with eigenvalue $0$.  Further,
the multiplicity of eigenvalue $0$ in spectrum of $\L(G)$ is equal to the number of connected components in graph $G$.     
Let us stress that our definition of $\L(G)$ has opposite sign with
respect to the one adopted by part of the computer science
literature. In particular, with the present definition, $\L(G)$ is a
positive semidefinite matrix.  

It is useful to recall a basic estimate on the Laplacian of random geometric graphs.
\begin{remark}
 \label{rem:spec_GNR}
 Let $\mathcal{L}_{n}$ denote the normalized Laplacian of the random geometric graph $G(n,r)$, defined as $\mathcal{L}_{n} = D^{-1/2} \mathcal{L} D^{-1/2}$, where $D$ is the diagonal matrix with degrees of the nodes on diagonal. Then, w.h.p.,
$\lambda_2(\mathcal{L}_{n})$, the second smallest eigenvalue of $\mathcal{L}_{n}$, is at least $C r^2$~(\cite{Boyd:mixingtimes,Penrose:RGG}). Also, using the result of~\cite{butler2008eas} (Theorem 4) and Corollary~\ref{cor:deg_GNR}, we have $\lambda_2(\mathcal{L}) \geq C(nr^d) r^2$, for some constant $C =C(d)$.
\end{remark}

\subsection{Notations}
For a vector $v \in \mathbb{R}^n$, and a subset $T \subseteq \{1,\cdots,n\}$, $v_T \in \reals^{T}$ is the restriction of $v$ to indices in $T$. We use the notation $\langle v_1,\cdots,v_n\rangle$ to represent the subspace spanned by vectors $v_i$, $1 \leq i \leq n$. The orthogonal projections onto subspaces $V$ and $V^{\perp}$ are respectively denoted by $P_V$ and $P^{\perp}_V $. The identity matrix, in any dimension, is denoted by $I$. Further, $e_i$ always refers to the $i^{th}$ standard basis element, e.g., $e_1 = (1,0,\cdots,0)$. Throughout this paper, $u \in \mathbb{R}^n$ is the all-ones vector and $C$ is a constant depending 
only on the dimension $d$, whose value may change from case to case. 

Given a matrix $A$, we denote its operator norm by $\|A\|_2$,
its Frobenius norm by $\|A\|_F$, its nuclear norm by $\|A\|_*$, and its $\ell_1$-norm by $\|A\|_{1}$.
($\|A\|_*$ is simply the sum of the singular values of $A$ and $\|A\|_1 = \sum_{ij} |A_{ij}|$). We also use $\sigma_{\max}(A)$ and $\sigma_{\min}(A)$ to respectively denote the maximum and the minimum nonzero singular values of $A$.

For a graph $G$, we denote by $V(G)$ the set of its vertices and we
use $E(G)$ to denote the set of edges in $G$. 
Following the convention adopted above, the Laplacian of $G$ is represented by $\L(G)$.

Finally, we denote by $x^{(i)}\in\reals^n$, $i\in \{1,\dots, d\}$ the $i^{th}$
column of the positions matrix $X$. In other words $x^{(i)}$ is the vector
containing the $i^{th}$ coordinate of points $x_1,\dots,x_n$.

Throughout the proof we shall adopt the convention of using the
notations $X$, $\{x_j\}_{j\in [n]}$, and $\{x^{(i)}\}_{i\in [d]}$ to
denote the centered positions. In other words $X=LX'$ where the 
rows of $X'$ are i.i.d. uniform in $[-0.5,0.5]^d$.
\section{Discussion}
\label{sec:discussion}

In this section, we make some remarks about Theorem~\ref{thm:main_result} and its implications.

\smallskip

\noindent{\bf Tightness of the Bounds.} The upper and the lower bounds in Theorem~\ref{thm:main_result} match up to the factor $C(nr^d)^5$. Note that $nr^d$ is the average degree of the nodes in $G$ (up to a constant) and when the rang $r$ is of the same order as the connectivity threshold, i.e., $r = O((\log n/n)^{1/d})$, it is logarithmic in $n$. Furthermore, we believe that this factor is the artifact of our analysis. The numerical experiments in Section~\ref{sec:simulation} also support the idea that the performance of the SDP-based algorithm, evaluated by $d(X,\hX)$, scales as $C\Delta/r^4$ for some constant $C$. In addition, the theorem states the bounds for $r \ge \alpha (\log n/n)^{1/d}$, with $\alpha \ge 10\sqrt{d}$. However, numerical experiments  in Section~\ref{sec:simulation} show that the bounds hold for much smaller $\alpha$, namely $\alpha \ge 3$ for $d = 2, 4$. Finally, it is immediate to see that under the worst case model for the measurement errors, no algorithm can perform better than $C \Delta/ r^2$. More specifically, for any algorithm $d(X,\hX) \ge C \Delta/r^2$, for some constant $C$. The reason is that letting $\tilde{d_{ij}}^2 = (1+\Delta/r^2) d_{ij}^2$, no algorithm can differentiate between $X$ and its scaled version $\hX = \sqrt{1+\Delta/r^2}\, X$. Also $d(X,\hX) = (\Delta/r^2) (1/n^2) \|LXX^TL\|_1 \ge C \Delta /r^2$, w.h.p. and  for some constant $C$ that depends on the dimension $d$. 

\smallskip

\noindent {\bf Global Rigidity of $G(n,r).$} 
As a  special case of Theorem~\ref{thm:main_result} we can consider
the problem of reconstructing the point positions from exact
measurements.  The case of exact measurements was also studied
recently in \cite{ShamsiYeTaheri} following a  different approach.
This corresponds to  setting $\Delta=0$. The underlying
question is whether the point positions $\{x_i\}_{i\in V}$ can be efficiently
determined (up to a rigid motion) by the set of distances
$\{d_{ij}\}_{(i,j)\in E}$. If this is the case, then, in particular,
the random graph $G(n,r)$ is globally rigid. 

Since the right-hand side of our error bound
Eq.~(\ref{eq:main_result1}) vanishes for $\Delta=0$, we immediately
obtain the following.
\begin{cor}
Let $\{x_1,\dots,x_n\}$ be $n$ nodes distributed uniformly at random
in the hypercube $[-0.5,0.5]^d$. If $r \geq 10 \sqrt{d} (\log n /
n)^{1/d}$, and the distance measurements are exact, then w.h.p., the
SDP-based algorithm recovers the exact positions (up to rigid
transformations). In particular, the random geometric graph $G(n,r)$
is w.h.p. \emph{globally rigid} if $r \geq 10 \sqrt{d} (\log n /
n)^{1/d}$.
\end{cor}
In~\cite{Aspnes2006}, the authors prove a similar result on global rigidity of $G(n,r)$. Namely, they show that if $n$ points are drawn from a Poisson process in $[0,1]^2$, then the random geometric graph $G(n,r)$ is globally rigid w.h.p. when $r$ is of the order $\sqrt{\log n/ n}$. 

As already mentioned above, the graph $G(n,r)$ is disconnected with
high probability if $r \le K_d^{-1/d} ((\log n +C)/n)^{1/d}$ for some
constant $C$. Hence, our result establishes the following
\emph{rigidity phase transition} phenomenon: There exist
dimension-dependent constants
$C_1(d)$, $C_2(d)$ such that a random geometric graph $G(n,r)$
is with high probability not globally rigid if $r\le C_1(d)(\log
n/n)^{1/d}$, and with high probability globally rigid if $r\ge C_2(d)(\log
n/n)^{1/d}$. Applying Stirling formula, it is easy to see that the
above arguments yield $C_1(d) \ge C_{1,*}\sqrt{d}$ and $C_{2}(d)\le
C_{2,*}\sqrt{d}$ for some numerical (dimension independent) constants
$C_{1,*}$, $C_{2,*}$.

It is natural to conjecture that the rigidity phase transition is
sharp.
\begin{conj}
Let $G(n,r_n)$ be a random geometric graph with $n$ nodes, and range
$r_n$, in $d$ dimensions. Then there exists a constant $C_*(d)$ such
that, for any $\ve>0$, the following happens.  If  $r_n \le (C_*(d)-\ve)(\log
n/n)^{1/d}$, then $G(n,r_n)$ is with high probability not globally
rigid.  If  $r_n \ge (C_*(d)+\ve)(\log
n/n)^{1/d}$, then $G(n,r_n)$ is with high probability globally
rigid. 
\end{conj}

\smallskip

\noindent{\bf Sensor Network Localization.}
Research in this area aims at developing algorithms
and systems to determine the positions of the nodes of a sensor network exploiting
inexpensive distributed measurements. Energy and hardware constraints
rule out the use of global positioning systems, and several proposed
systems exploit pairwise distance measurements between the sensors
\cite{SensorNetworks1,SensorNetworks2}. 
These techniques have acquired new industrial interest due to their
relevance to indoor positioning. In this context, global positioning
systems are not a method of choice because of their limited accuracy
in indoor environments. 

Semidefinite programming methods for sensor network localization have been
developed starting with \cite{Biswas:SDP}.
It is common to study and evaluate different
techniques within the random geometric graph model,
but no performance guarantees have been proven for advanced (SDP
based) algorithms, with inaccurate measurements.
We shall therefore consider $n$
sensors placed uniformly at random in the unit hypercube, 
with ambient dimension  either $d=2$ or $d=3$ depending on the
specific application. The connectivity range $r$ is dictated by
various factors: power limitations; interference between nearby nodes;
loss of accuracy with distance. 
 
The measurement error $z_{ij}$ depends on the method used to measure
the distance between nodes $i$ and $j$. We will limit ourselves to
measurement errors due to  noise (as opposed --for instance-- to
malicious behavior of the nodes) and discuss two common techniques for
measuring distances between wireless devices: Received Signal Indicator (RSSI) and Time 
Difference of Arrival (TDoA). RSSI measures the ratio of the power
present in a received radio signal $(P_r)$ and a reference 
transmitted power $(P_s)$. The ratio $P_r / P_s$ is inversely
proportional to the square of the distance between the receiver and the transmitter. Hence, RSSI can be used to
 estimate the distance.
It is reasonable to assume that the dominant error is in the
measurement of the received power, and that it is proportional to
the   transmitted power.
We thus assume that there is an error $\ve\, P_s$ in measuring the
received power $P_r$., i.e., $\widetilde{P}_r = P_r + 
\ve \, P_s$, where $\widetilde{P}_r$ denotes the measured received power. Then,
the measured distance is given by
\begin{equation}
\tilde{d}_{ij}^2 \propto \frac{P_s}{\tilde{P}_r} = \frac{P_s}{P_r}
\cdot\Big(
1+\frac{P_s}{P_r}\, \ve\Big)^{-1}
 \approx \frac{P_s}{P_r}\Big (1 - \frac{P_s}{P_r}\ve\Big) \propto d_{ij}^2 (1  -C d^2_{ij}\ve).
\end{equation} 
Therefore the overall error $|z_{ij}| \propto d_{ij}^4\ve$ and 
its magnitude is $\Delta \propto r^4\ve$. Applying
Theorem~\ref{thm:main_result},
we obtain an average error per node of order
\begin{eqnarray*}
d(X,\hX)\le C_1'(nr^d)^5\, \ve\, .
\end{eqnarray*}
In other words, the positioning accuracy is linear in the measurement
accuracy, with a proportionality constant that is polynomial in the
average node degree. Remarkably, the best accuracy is obtained by using
the smallest average degree, i.e. the smallest measurement radius that
is compatible with connectivity. 

 TDoA technique uses the time difference
between the receipt of two different signals with different
velocities, for instance ultrasound and radio signals. The
time difference is proportional to the distance between the receiver
and the transmitter,  and given the velocity of the signals the
distance can be estimated from the time difference.
 Now, assume that there is a relative  error $\ve$ in measuring this time
 difference (this might be related to inaccuracies in ultrasound speed). 
We thus have $\widetilde{t}_{ij}=t_{ij}(1+\ve)$, where
$\widetilde{t}_{ij}$ is the measured time while $t_{ij}$ is the
`ideal' time difference.  This leads to an error in estimating $d_{ij}$ which
 is proportional to $d_{ij}\ve$. 
Therefore, $|z_{ij}| \propto d^2_{ij}\ve$ and $\Delta \propto r^2\ve$.
 Applying again Theorem~\ref{thm:main_result},
we obtain an average error per node of order
\begin{eqnarray*}
d(X,\hX)\le C_1'(nr^d)^5\, \frac{\ve}{r^2}\, .
\end{eqnarray*}
In other words the reconstruction error decreases with the measurement
radius, which suggests somewhat different network design for such a
system.

Let us stress in passing that the above error bounds are proved under
an adversarial error model (see below). It would be useful to
complement them with similar analysis carried out for other, more
realistic, models.

\smallskip
  
\noindent{\bf Manifold Learning.} Manifold learning deals with finite data sets
of points in ambient space $\reals^N$ which are assumed to lie on a
smooth submanifold $\mathcal{M}^d$ of  dimension $d < N$. The task is
to recover $\mathcal{M}$ given only the data points.
Here, we discuss  the implications of
Theorem~\ref{thm:main_result}  for applications of SDP methods to
manifold learning.

It is typically assumed that the manifold $\mathcal{M}^d$ 
is isometrically equivalent to a region in $\reals^d$. For the sake of
simplicity we
shall assume that this region is convex (see~\cite{DonohoGrimes} for a discussion of this point). With little loss of
generality  we can indeed identify the region with the unit 
hypercube $[-0.5,0.5]^d$. 
A typical manifold learning algorithm (\cite{Tenenbaum:ISOMAP} and~\cite{Weinberger06anintroduction})
estimates the \emph{geodesic} distances between a subset of pairs of 
data points
$d_{\mathcal{M}} (y_i,y_j)$, $y_i\in\reals^N$, and then tries to find a
low-dimensional embedding (i.e. positions $x_i\in\reals^d$) that
reproduce these distances.

The unknown geodesic distance
between nearby data points $y_i$ and $y_j$, denoted by
$d_{\mathcal{M}} (y_i,y_j)$, can be estimated by their Euclidean
distance in $\reals^n$. Therefore the manifold learning problem
reduces mathematically to the localization problem whereby the
distance `measurements' are $\tilde{d}_{ij} = \|y_i -
y_j\|_{\reals^N}$, while the actual distances are $d_{ij} = d_{\mathcal{M}} (y_i,y_j)$.
The accuracy of these estimates depends on the curvature of the
manifold $\mathcal{M}$.
 Let $r_0  = r_0(\mathcal{M})$ be the \emph{minimum radius of curvature} defined by:
\begin{equation*}
\frac{1}{r_0} = \max_{\gamma,t} \{\| \ddot \gamma(t)\|\},
\end{equation*}
where $\gamma$ varies over all unit-speed geodesics in $\mathcal{M}$
and $t$ is in the domain of $\gamma$. 
For instance, an Euclidean sphere of radius $r_0$ has  minimum radius
of curvature equal to $r_0$. 

As shown in~\cite{Bernstain:manifold} (Lemma 3), $(1 - d_{ij}^2 / 24 r_0^2) d_{ij} \le \tilde{d}_{ij} \le d_{ij}$. Therefore, $|z_{ij}| \propto d_{ij}^4 / r_0^2$, and $\Delta \propto r^4/r_0^2$. Theorem~\ref{thm:main_result} supports the claim that the estimation error $d(X, \hX)$ is bounded by $C(nr^d)^5 / r_0^2$. 

\smallskip

\smallskip

As mentioned several times, this paper focuses on a particularly
simple SDP relaxation, and noise model. This opens the way to a number
of interesting directions:
\begin{enumerate}
\item \emph{Stochastic noise models}. A somewhat complementary
  direction to the one taken here would be to assume that the distance
  measurements are $\tilde{d}_{ij}^2=d_{ij}^2+z_{ij}$ with
  $\{z_{ij}\}$ a collection of independent zero-mean random
  variables. This would be a good model, for instance, for errors in
  RSSI measurements. 

  Another interesting case would be the one in which a small subset of 
  measurements are grossly incorrect (e.g. due to node malfunctioning,
  obstacles, etc.).
\item \emph{Tighter convex relaxations}. The relaxation considered here 
is particularly simple, and can be improved in several ways. For
instance, in manifold learning it is useful to maximize the 
embedding variance $\Tr(Q)$ under the constraint $Qu=0$
\cite{Weinberger06anintroduction}.

Also, for any pair $(i,j)\not\in E$ it is possible to add a
constraint of the form 
$\langle M_{ij},Q\rangle\le \hat{d}_{ij}^2$, where $\hat{d}_{ij}$ is an upper
bound on the distance obtained by computing the shortest path between
$i$ and $j$ in $G$.
\item \emph{More general geometric problems}.
The present paper analyzes the problem of reconstructing the geometry
of a cloud of points from incomplete and inaccurate measurements of
the points local geometry. From this point of view, a number of interesting
extensions can be explored. For instance, instead of
distances, it might be possible to measure angles between edges in the
graph $G$ (indeed in sensor networks, angles of arrival might be
available \cite{SensorNetworks1,SensorNetworks2}). 
\end{enumerate}

\section{Proof of Theorem~\ref{thm:main_result} (Upper Bound)}
\label{sec:main_thm_proof}

 Let $V = \langle u, x^{(1)},\cdots, x^{(d)}\rangle$ and for any matrix $S \in \mathbb{R}^{n \times n}$, define
\begin{equation}
\label{def:orth_part}
\tilde{S} = P_V SP_V + P_V S P^{\perp}_V + P^{\perp}_V S P_V, \quad \quad S^{\perp} = P^{\perp}_V S P^{\perp}_V \, .
\end{equation}  
Thus $S = \tilde{S} + S^{\perp}$. Also, denote by $R$ the difference between the optimum solution $Q$ and the actual Gram matrix $Q_0$, i.e., $R = Q - Q_0$. The 
proof of Theorem \ref{thm:main_result} is based on the following
key lemmas  that bound $R^{\perp}$ and $\tilde{R}$ separately. 
\begin{lem}
\label{lem:bound on Rperp}
There exists a numerical constant $C = C(d)$, such that, w.h.p.,
\begin{equation}
\|R^{\perp}\|_* \leq C \frac{n}{r^4} (nr^d)^5 \Delta.
\end{equation}
\end{lem} 

\begin{lem}
\label{lem:bound on Rtilde}
There exists a numerical constant $C = C(d)$, such that, w.h.p.,
\begin{equation}
\|\tilde{R}\|_1 \leq C \frac{n^2}{r^4} (nr^d)^5 \Delta.
\end{equation}
\end{lem}
We defer the proof of lemmas~\ref{lem:bound on Rperp} and~\ref{lem:bound on Rtilde} to the next section.

\begin{proof}[Proof (Theorem~\ref{thm:main_result})] Let $Q = \sum_{i=1}^n \sigma_i u_i u_i^T$, where $\|u_i\|=1$, $u_i^Tu_j = 0$ for $i \neq j$ and $\sigma_1 \geq \sigma_2\geq \cdots \geq \sigma_n \ge0$. Let $\cP_d(Q) = \sum_{i=1}^d \sigma_i u_i u_i^T$ be the best rank-$d$ approximation of $Q$ in Frobenius norm (step 2 in the algorithm). Recall that $Qu = 0$, because $Q$ minimizes $\Tr(Q)$. Consequently, $\cP_d(Q) u = 0$ and $\cP_d(Q) = L\cP_d(Q)L$. Further, by our assumption $Q_0 u =0$ and thus $Q_0 = LQ_0L$. Using triangle inequality,
\begin{align}
\label{eq:bound1}
\|L\cP_d(Q)L - LQ_0L&\|_1 = \|\cP_d(Q) - Q_0\|_1 \nonumber\\
& \leq \|\cP_d(Q) - \tilde{Q}\|_1 + \|\tilde{Q} - Q_0\|_1.
\end{align}
Observe that, $\tilde{Q} = Q_0 + \tilde{R}$ and $Q^{\perp} = R^{\perp}$. Since $\cP_d(Q) - \tilde{Q}$ has rank at most $3d$, it follows that
$\|\cP_d(Q) - \tilde{Q}\|_1 \leq n \|\cP_d(Q) - \tilde{Q}\|_F \leq \sqrt{3 d} n \|\cP_d(Q) - \tilde{Q}\|_2$
(for any matrix $A$, $\|A\|_F^2 \leq \rank(A) \|A\|_2^2$).
By triangle inequality, we have
\begin{equation}
\label{eq:bound3}
\|\cP_d(Q) - \tilde{Q}\|_2 \leq \|\cP_d(Q) - Q\|_2 + \|\underbrace{Q-\tilde{Q}}_{R^{\perp}}\|_2.
\end{equation} 
Note that $\|\cP_d(Q) - Q\|_2 = \sigma_{d+1}$. Recall the variational principle for the eigenvalues.
\begin{equation*}
\sigma_q = \min_{H,dim(H) = n- q +1} \max_{y \in H, \|y\|=1} y^TQy.
\end{equation*}
Taking $H = \langle x^{(1)},\cdots,x^{(d)}\rangle^{\perp}$, for any $y \in H$,
 $y^TQy = y^T P^{\perp}_V Q P^{\perp}_Vy = y^T Q^{\perp} y = y^T R^{\perp} y$,
 where we used the fact $Qu = 0$ in the first equality. 
Therefore,
 $\sigma_{d+1} \leq \max_{\|y\|=1} y^T R^{\perp} y = \|R^{\perp}\|_2$ 
 It follows from Eqs.~\eqref{eq:bound1} and~\eqref{eq:bound3} that
 \begin{equation*}
 \|L\cP_d(Q)L - LQ_0L\|_1 \leq
 2 \sqrt{3d} n \|R^{\perp}\|_2 + \|\tilde{R}\|_1.
 \end{equation*}
Using Lemma~\ref{lem:bound on Rperp} and~\ref{lem:bound on Rtilde}, we obtain
\begin{equation*}
d(X,\hX) = \frac{1}{n^2} \|L\cP_d(Q)L-LQ_0L\|_1 \leq C (n r^d)^5 \frac{\Delta}{r^4},
\end{equation*}
which proves the claimed upper bound on the error.

The lower bound is proved in Section~\ref{sec:converse}.
\end{proof}
%
\section{Proof of Lemma~\ref{lem:bound on Rperp}}
\label{sec:bound on Rperp}
The proof is based on the following three steps: $(i)$ Upper bound $\|R^{\perp}\|_*$ in terms of $\sigma_{\min}(\Omega)$ and $\sigma_{\max}(\Omega)$, where $\Omega$ is an arbitrary positive semidefinite (PSD) stress matrix of rank $n-d-1$ for the framework; $(ii)$ Construct a particular PSD stress matrix $\Omega$ of rank $n-d-1$ for the framework; $(iii)$ Upper bound $\sigma_{\max}(\Omega)$ and lower bound $\sigma_{\min}(\Omega)$.

\begin{thm}
\label{thm:Rtilde_bound}
Let $\Omega$ be an arbitrary PSD stress matrix for the framework $G_X$ such that $\rank(\Omega) = n -d -1$. Then,
\begin{equation}
\label{eq:Rperp_bound}
\|R^{\perp}\|_* \leq 2 \frac{\sigma_{\max}(\Omega)}{\sigma_{\min}(\Omega)} |E| \Delta.
\end{equation}
\end{thm}

\begin{proof}
Note that $R^{\perp} = Q^{\perp} = P^{\perp}_V Q P^{\perp}_V \succeq \bf{0}$. Write $R^{\perp} = \sum_{i=1}^{n -d -1} \lambda_i u_i u_i^T$, where $\|u_i\| = 1$, $u_i^T u_j = 0$ for $i \neq j$ and $\lambda_1 \geq \lambda_2 \geq \cdots \lambda_{n-d-1} \geq 0$. Therefore,
\begin{equation}
\label{eq:lb_OmegaRperp}
\langle \Omega, R^{\perp} \rangle = 
\langle \Omega, \sum_{i=1}^{n-d-1} \lambda_i\; u_i u_i^T \rangle =
\sum_{i=1}^{n-d-1} \lambda_i u_i^T \Omega u_i \geq \sigma_{\min}(\Omega) \|R^{\perp}\|_*.
\end{equation}
Here, we used the fact that $u_i \in V^{\perp} = \text{Ker}^{\perp}(\Omega)$. Note that $\sigma_{\min}(\Omega) > 0$, since $\Omega \succeq \bf{0}$.

Now, we need to upper bound the quantity $\langle \Omega, R^{\perp}\rangle$. Since $\Omega u = 0$, the stress matrix $\Omega = [\omega_{ij}]$ can be written as $\Omega = \sum_{(i,j) \in E} \omega_{ij} M_{ij}$. Define $\omega_{\max} = \underset{i \neq j}{\max} |\omega_{ij}|$. Then,
\begin{align}
\langle \Omega, R^{\perp} \rangle &\stackrel{(a)}{=}  \langle \Omega, R \rangle = \sum_{(i,j) \in E} \omega_{ij} \langle M_{ij}, R\rangle \nonumber\\
& \leq \sum_{(i,j) \in E} \omega_{\max} |\langle M_{ij}, Q - Q_0\rangle| \nonumber\\
& \leq \sum_{(i,j) \in E} \omega_{\max} (|\langle M_{ij}, Q \rangle - \tilde{d_{ij}}^2| + |\underbrace{\tilde{d_{ij}}^2 - d^2_{ij}}_{z_{ij}}|) \nonumber\\
&\leq 2 \omega_{\max} |E| \Delta, \label{eq:ub_OmegaRperp}
\end{align}
where $(a)$ follows from the fact that $\langle P_V,\Omega\rangle = 0$. Since $\Omega \succeq \bf{0}$, we have $\omega_{ij}^2 \leq \omega_{ii}  \omega_{jj} = (e_i^T \Omega e_i) (e_j^T \Omega e_j) \leq \sigma_{\max}^2 (\Omega)$, for $1 \leq i,j \leq n$. Hence, $\omega_{\max} \leq \sigma_{\max}(\Omega)$. Combining Eqs.~\eqref{eq:lb_OmegaRperp} and~\eqref{eq:ub_OmegaRperp}, we get the desired result.
\end{proof}

Next step is constructing a PSD stress matrix of rank $n-d-1$.
For each node $i \in V(G)$ define \mbox{$\mathcal{C}_i = \{j \in V(G) : d_{ij} \leq r / 2\}$}. Note that the nodes in each $\mathcal{C}_i$ form a clique in $G$. In addition, let $S_i$ be the following set of cliques. 
\begin{equation*}
S_i :=  \underset{k \in \Ci}{\cup} \{\Ci \backslash k\} \cup \{\Ci \}. 
\end{equation*}
Therefore, $S_i$ is a set of $|\Ci|+1$ cliques. For the graph $G$, we define $\cliq(G) := S_1\cup \cdots \cup S_n$. Next lemma establishes a simple property of cliques $\C_i$. Its proof is immediate and deferred to Appendix~\ref{App:sampling_lemma}.
\begin{pro}
\label{pro:sampling_lemma}
If $r \ge 4 c \sqrt{d}(\log n /n)^{1/d}$ with $c > 1$, the following is true w.h.p.. For any two nodes $i$ and $j$, such that $\|x_i -x_j\| \leq r/2$, $| \mathcal{C}_i \cap \mathcal{C}_j| \geq d+1$. 
\end{pro}

Now we are ready to construct a special stress matrix $\Omega$ of $G_X$. Define the $|\Q_k| \times |\Q_k|$ matrix $\Omega_k$ as follows. 
\begin{equation*}
\Omega_k =  P^{\perp}_{\langle u^{}_{\mathcal{Q}_k},x^{(1)}_{\mathcal{Q}_k},\cdots,x^{(d)}_{\mathcal{Q}_k}\rangle}.
\end{equation*}
Let $\hat{\Omega}_k$ be the $n \times n$ matrix obtained from $\Omega_k$ by padding it with zeros. Define 
\begin{equation*}
\Omega = \underset{\mathcal{Q}_k \in \cliq(G)}{\sum} \hat{\Omega}_k.
\end{equation*}
The proof of the next statement is again immediate and discussed in Appendix~\ref{App:stress_matrix}.
\begin{pro}
\label{pro:stress_construction}
The matrix $\Omega$ defined above is a positive semidefinite (PSD) stress matrix for the framework $G_X$. Further, almost surely, $\rank(\Omega) = n-d-1$.
\end{pro}

Final step is to upper bound $\sigma_{\max}(\Omega)$ and lower bound $\sigma_{\min}(\Omega)$.

\begin{claim}
\label{claim:sigma_max_omega}
There exists a constant $C = C(d)$, such that, w.h.p.,
\begin{equation*}
\sigma_{\max}(\Omega) \leq C (n r^d)^2.
\end{equation*}
\end{claim}
\begin{proof}
For any vector $v \in \mathbb{R}^n$,
\begin{align*}
 v^T \Omega v &= \sum_{\Q_k \in \cliq(G)} v^T \hat{\Omega}_k v = \sum_{\Q_k \in \cliq(G)} \| \hat{\Omega}_k v\|^2 =  \underset{\Q_k \in \cliq(G)}{\sum} \| P^{\perp}_{\langle u_{\Q_k},x^{(1)}_{\Q_k},\cdots,x^{(d)}_{\Q_k}\rangle} v_{\Q_k}\|^2 \\
 & \leq \sum_{\Q_k \in \cliq(G)} \|v_{\Q_k}\|^2 
 = \sum_{j=1}^n v^2_j \sum_{k: j \in \Q_k} 1 = \sum_{j=1}^n (\sum_{i \in \Cj} |\Ci|) v_j^2  \leq (C nr^d  \|v\|) ^2. 
 \end{align*}
 The last inequality follows from the fact that, w.h.p., $|\mathcal{C}_j| \leq C nr^d$ for all $j$ and some constant $C$ (see Corollary~\ref{cor:deg_GNR}).
\end{proof}

We now pass to lower bounding $\sigma_{\min}(\Omega)$.


\begin{thm}
\label{thm:OmegaL}
There exists a constant $C=C(d)$, such that, w.h.p., $\Omega^{\perp} \succeq C(nr^d)^{-3} r^2 \mathcal{L}^{\perp}$. \textup{(}see Eq.~\eqref{def:orth_part}\textup{)}. 
\end{thm}

The proof is given in Section~\ref{sec:OmegaL}. We are finally in position to prove Lemma~\ref{lem:bound on Rperp}.
\begin{proof}[Proof (Lemma~\ref{lem:bound on Rperp})]
Following Theorem~\ref{thm:OmegaL} and Remark~\ref{rem:spec_GNR}, we obtain $\sigma_{\min}(\Omega) \geq C(nr^d)^{-2}r^4$. Also, by Corollary~\ref{cor:deg_GNR}, w.h.p., the node degrees in $G$ are bounded by $3/2 K_d nr^d$. Hence, w.h.p., $|E| \leq 3/4 n^2 K_d r^d$.
Using the bounds on $\sigma_{\max}(\Omega)$, $\sigma_{\min}(\Omega)$ and $|E|$ in Theorem~\ref{thm:Rtilde_bound} yields the thesis. 
\end{proof}

\subsection{Proof of Theorem~\ref{thm:OmegaL}}
\label{sec:OmegaL}
Before turning to the proof, it is worth mentioning that the authors in~\cite{Belkin02laplacianeigenmaps} propose a heuristic argument showing $\Omega v \approx \L^2 v$ for smoothly varying vectors $v$. Since $\sigma_{\min}(\L) \geq C(nr^d) r^2$ (see Remark~\ref{rem:spec_GNR}), this heuristic supports the claim of the theorem.

In the following, we first establish some claims and definitions which will be used in the proof.

\begin{claim}
\label{claim:Laplacian}
There exists a constant $C=C(d)$, such that, w.h.p.,
\begin{equation*}
\mathcal{L} \preceq C \sum_{k=1}^n P_{u^{}_{\mathcal{C}_k}}^{\perp}.
\end{equation*}
\end{claim}
The argument is closely related to the Markov chain comparison technique~\cite{Persi:Comparison}. The proof is given in Appendix~\ref{App:Laplacian}.

The next claim provides a concentration result about the number of nodes in the cliques $\Ci$. Its proof is immediate and deferred to Appendix~\ref{App:concentration_points}.

\begin{claim}
\label{claim:concentration_points}
For every node $i \in V(G)$, define $\tilde{\Ci} = \{j \in V(G): d_{ij} \leq \frac{r}{2}(\frac{1}{2} + \frac{1}{100})\}$. There exists an integer number $m$ such that the following is true w.h.p..
\begin{equation*}
|\tilde{\Ci}| \leq m \leq |\Ci|, \quad \forall i \in V(G).
\end{equation*}
\end{claim}

Now, for any node $i$, let $i_1,\cdots i_m$ denote the $m$-nearest neighbors of that node. Using claim~\ref{claim:concentration_points}, $\tilde{\Ci} \subseteq \{i_1,\cdots,i_m\} \subseteq \Ci$. Define the set $\tilde{S}_i$ as follows.
\begin{equation*}
\tilde{S}_i = \{\Ci, \Ci \backslash{i_1},\cdots,\Ci \backslash{i_m}\}.
\end{equation*}
Therefore, $\tilde{S}_i$ is a set of $(m+1)$ cliques. Let $\cliq^*(G) = \tilde{S}_1\cup \cdots \cup \tilde{S}_n$. Note that $\cliq^*(G) \subseteq \cliq(G)$. Construct the graph $G^*$ in the following way. For every element in $\cliq^*(G)$, there is a corresponding vertex in $G^*$. (Thus, $|V(G^*)| = n(m+1)$). Also, for any two nodes $i$ and $j$, such that $\|x_i - x_j\| \leq r/2$, every vertex in $V(G^*)$ corresponding to an element in $\tilde{S}_i$ is connected to every vertex in $V(G^*)$ corresponding to an element in $\tilde{S}_j$.

Our next claim establishes some properties of the graph $G^*$. For its proof, we refer to Appendix~\ref{App:Gstar}.

\begin{claim}
\label{claim:Gstar}
With high probability, the graph $G^*$ has the following properties.
\begin{itemize}
\item[$(i)$] The degree of each node is bounded by $C (nr^d)^2$, for some constant $C = C(d)$.
 
\item[$(ii)$] Let $\mathcal{L}^*$ denote the Laplacian of $G^*$. Then $\sigma_{\min}(\mathcal{L}^*) \geq C(nr^d)^2 r^2$, for some constant $C$.
\end{itemize}
\end{claim} 

Now, we are in position to prove Theorem~\ref{thm:OmegaL}

\begin{proof}[Proof (Theorem~\ref{thm:OmegaL})]
Let $v \in V^{\perp}$ be an arbitrary vector. For every clique $\Q_i \in \cliq(G)$, decompose $v$ locally as $v_{\Q_i} = \sum_{\ell =1}^{d} \beta_i^{(\ell)} \tilde{x}_{\Q_i}^{(\ell)} + \gamma_i u^{}_{\Q_i} + w^{(i)}$, where $\tilde{x}_{\Q_i}^{(\ell)} = P_{u^{}_{\Q_i}}^{\perp} x_{\Q_i}^{(\ell)}$ and $w^{(i)} \in \langle x_{\Q_i}^{(1)}, \cdots, x_{\Q_i}^{(d)}, u_{\Q_i}\rangle ^{\perp}$. Hence, 

\begin{equation*}
v^T \Omega v = \sum_{\Q_i \in \cliq(G)} \|w^{(i)}\|^2.
\end{equation*}

\noindent Note that $v_{\Q_i \cap \Q_j}$ has two representations; One is obtained by restricting $v_{\Q_i}$ to indices in $\Q_j$, and the other is obtained by restricting $v_{\Q_j}$ to indices in $\Q_i$. From these two representations, we get 
\begin{equation}
\label{eq:representation}
w^{(i)}_{\Q_i \cap \Q_j} - w^{(j)}_{\Q_i \cap \Q_j} = 
\sum_{\ell=1}^{d} (\beta_j^{(\ell)} -\beta_i^{(\ell)}) \tilde{x}_{\Q_i \cap \Q_j}^{(\ell)} + \tilde{\gamma}_{i,j}  u_{\Q_i \cap \Q_j}.
\end{equation}
Here, $\tilde{x}_{\Q_i \cap \Q_j}^{(\ell)} = P_{u^{}_{\Q_i \cap \Q_j}}^{\perp} x_{\Q_i \cap \Q_j}^{(\ell)}$. The value of $\gamma_{i,j}$ does not matter to our argument; however it can be given explicitly. 


Note that $\{\C_1,\cdots,\C_n\} \subseteq \cliq^*(G) \subseteq \cliq(G)$. Invoking Claim~\ref{claim:Laplacian}, 
\begin{align*}
v^T \mathcal{L} v \leq C \sum_{k=1}^n \|P_{u^{}_{\mathcal{C}_k}}^{\perp} v^{}_{\mathcal{C}_k}\| ^2 &\leq C \sum_{\Q_i \in \cliq^*(G)} \|P_{u^{}_{\Q_i}}^{\perp} v^{}_{\Q_i}\|^2
= C \sum_{\Q_i \in \cliq^*(G)} \Big\|\sum_{\ell=1}^{d} \beta_i^{(\ell)} \tilde{x}_{\Q_i}^{(\ell)} + w^{(i)}\Big\|^2\\ 
&= C\left( \sum_{\Q_i \in \cliq^*(G)} \Big\| \sum_{\ell=1}^{d} \beta_i^{(\ell)} \tilde{x}_{\Q_i}^{(\ell)} \Big\|^2 + \sum_{\Q_i \in \cliq^*(G)} \| w ^{(i)} \|^2 \right)\\
&\leq C\left( d \sum_{\Q_i \in \cliq^*(G)} \sum_{\ell=1}^{d} \Big\| \beta_i^{(\ell)}\tilde{x}_{\Q_i}^{(\ell)} \Big\|^2 + \sum_{\Q_i \in \cliq^*(G)} \| w ^{(i)} \|^2 \right).
\end{align*}
Hence, we only need to show that
\begin{equation}
\label{eq:aux1}
\sum_{\Q_i \in \cliq(G)} \|w^{(i)}\|^2 \geq 
C(nr^d)^{-3} r^2 \sum_{\Q_i \in \cliq^*(G)} \sum_{\ell=1}^{d} \Big\| \beta_i^{(\ell)}\tilde{x}_{\Q_i}^{(\ell)} \Big\|^2,
\end{equation}
for some constant $C = C(d)$.

In the following we adopt the convention that for $j \in V(G^*)$, $\Q_j$ is the corresponding clique in $\cliq^*(G)$. We have

\begin{align}
\sum_{\Q_i \in \cliq(G)} \|w^{(i)}\|^2 &\geq \sum_{\Q_i \in \cliq^*(G)} \|w^{(i)}\|^2  = 
\sum_{i \in V(G^*)} \|w^{(i)}\|^2 \nonumber\\
&\stackrel{(a)} \geq C(nr^d)^{-2} \sum_{(i,j)\in E(G^*)}(\|w^{(i)}\|^2 + \|w^{(j)}\|^2) \nonumber\\
&\geq C(nr^d)^{-2} \sum_{(i,j)\in E(G^*)}\|w^{(i)}_{\Q_i \cap \Q_j} - w^{(j)}_{\Q_i \cap \Q_j}\|^2 \nonumber\\
& \stackrel{(b)}\geq C(nr^d)^{-2} \sum_{(i,j)\in E(G^*)} \Big\| \sum_{\ell=1}^{d} (\beta_j^{(\ell)} -\beta_i^{(\ell)}) \tilde{x}_{\Q_i \cap \Q_j}^{(\ell)} \Big\|^2 \nonumber\\
&\stackrel{(c)} \geq C(nr^d)^{-1} r^2 \sum_{(i,j)\in E(G^*)} \sum_{\ell=1}^{d} (\beta_j^{(\ell)} -\beta_i^{(\ell)})^2. \label{eq:LB_stress}
\end{align}
Here, $(a)$ follows from the fact that the degrees of nodes in $G^*$ are bounded by $C(nr^d)^2$ (Claim~\ref{claim:Gstar}, part $(i)$); $(b)$ follows from Eq.~\eqref{eq:representation} and $(c)$ follows from Claim~\ref{claim:Azuma_bound}, whose proof is deferred to Appendix~\ref{App:Azuma_bound}.
\begin{claim}
\label{claim:Azuma_bound}
There exists a constant $C = C(d)$, such that, for any set of values $\{\beta_i^{(\ell)}\}$ the following holds with high probability.
\begin{equation*}
\Big\| \sum_{\ell=1}^{d} (\beta_j^{(\ell)} -\beta_i^{(\ell)}) \tilde{x}_{\Q_i \cap \Q_j}^{(\ell)} \Big\|^2
\geq  
C (nr^d) r^2 \sum_{\ell=1}^{d} (\beta_j^{(\ell)} -\beta_i^{(\ell)})^2, \quad \forall (i,j) \in E(G^*).
\end{equation*}
\end{claim}

Also note that $\|\txl_{\Q_i}\|^2 \le |\Q_i| r^2$ and w.h.p., $|\Q_i| \le C(nr^d)$, for all $i \in V(G^*)$ (since, w.h.p., $|\C_i| \le C(nr^d)$ for all $i \in V(G)$ by Corollary~\ref{cor:deg_GNR} ). Therefore, using Eq.~\eqref{eq:LB_stress}, in order to prove~\eqref{eq:aux1} it suffices to show that
\begin{equation*}
\sum_{\ell=1}^{d} \sum_{(i,j)\in E(G^*)}(\beta_j^{(\ell)} - \beta_i^{(\ell)})^2 \geq
C(nr^d)^{-1} r^2 \sum_{\ell=1}^{d} \sum_{i \in V(G^*)} (\beta_i^{(\ell)})^2.
\end{equation*} 
Define $\beta^{(\ell)} = (\beta_i^{(\ell)})_{i \in V(G^*)}$. Observe that,
\begin{equation*}
\sum_{(i,j)\in E(G^*)}(\beta_j^{(\ell)} - \beta_i^{(\ell)})^2 = (\beta^{(\ell)})^T \mathcal{L}^* \beta^{(\ell)} \geq \sigma_{\min}(\mathcal{L}^*) \|P^{\perp}_u \beta^{(\ell)}\|^2.
\end{equation*}
Using Claim~\ref{claim:Gstar} (part $(ii)$) we obtain
\begin{equation*}
\sum_{(i,j)\in E(G^*)}(\beta_j^{(\ell)} - \beta_i^{(\ell)})^2
\geq C(nr^d)^2r^2  \|P^{\perp}_u \beta^{(\ell)}\|^2.
\end{equation*}
The proof is completed by the following claim, whose proof is given in Appendix~\ref{App:beta_mean}.
 \end{proof}
\begin{claim}
\label{claim:beta_mean}
There exists a constant $C =C(d)$, such that, the following holds with high probability. Consider an arbitrary vector $v \in V^{\perp}$ with local decompositions $v_{\Q_i} = \sum_{\ell=1}^{d} \beta^{(\ell)}_i \tilde{x}^{(\ell)}_{\Q_i} + \gamma_i u_{\Q_i} + w^{(i)}$. Then,
\begin{equation*}
\sum_{\ell=1}^{d} \|P_u^{\perp} \beta^{(\ell)}\|^2 \geq C(nr^d)^{-3} \sum_{\ell=1}^{d} \|\beta^{(\ell)}\|^2.
\end{equation*}
\end{claim}

%
%
\section{Proof of Lemma~\ref{lem:bound on Rtilde}}
\label{sec:bound on Rtilde}
Recall that $\tilde{R} = P_VRP_V + P_V R P_V^{\perp} + P_V^{\perp} R P_V$, and $V = \langle x^{(1)},\cdots,x^{(d)}, u\rangle$. Therefore, there exist a matrix $Y \in \mathbb{R}^{n \times d}$ and a vector $a \in \mathbb{R}^n$ such that $\tilde{R} = XY^T + YX^T + u a^T + a u^T$. We can further assume that $Y^Tu = 0$. Otherwise, define $\tilde{Y} = Y - u (u^T Y / \|u\|^2)$ and $\tilde{a} = a + X(Y^T u / \|u\|^2)$. Then $\tR = X\tilde{Y}^T + \tilde{Y} X^T + u \tilde{a}^T + \tilde{a} u^T$, and $\tilde{Y}^T u = 0$. 

Also note that, $u^TQu = u^T \tR u = 2 (a^Tu) \|u\|^2$. Hence, $a^Tu = 0$, since $Qu = 0$. In addition, $Qu = \tR u =a \|u\|^2$, which implies that $a = 0$. Therefore, $\tR = XY^T + YX^T$ where $Y^Tu = 0$. Denote by $y_i^T \in \mathbb{R}^d$, $i\in [n]$, the $i^{th}$ row of the matrix $Y$.

Define the operator $\cR_{G,X} : \reals^{n \times d} \to \reals^E$ as $\cR_{G,X}(Y) = R_G(X) \mathcal{Y}$, where $\mathcal{Y} = [y_1^T, \cdots,y_n^T]^T$ and $R_G(X)$ is the rigidity matrix of framework $G_X$. Observe that 
\begin{equation*}
\|\cR_{G,X}(Y)\|_1 = \sum_{(l,k) \in E(G)} |\langle x_l-x_k, y_l - y_k\rangle|.
\end{equation*}
 
The following theorem compares the operators $\cR_{G,X}$ and $\cR_{K_n,X}$, where $G = G(n,r)$ and $K_n$ is the complete graph with $n$ vertices. This theorem is the key ingredient in the proof of Lemma~\ref{lem:bound on Rtilde}.
\begin{thm}
\label{thm:cheeger}
There exists a constant $C=C(d)$, such that, w.h.p.,
\begin{equation*}
\|\cR_{K_n,X}(Y)\|_1 \le C r^{-d-2} \|\cR_{G,X}(Y)\|_1, \quad \text{for all }Y \in \reals^{n \times d}.
\end{equation*}
\end{thm} 

Proof of Theorem~\ref{thm:cheeger} is discussed in next subsection. The next statement provides an upper bound on $\|\tilde{R}\|_1$. Its proof is immediate and discussed in Appendix~\ref{App:aux2}.

\begin{pro}
\label{pro:aux2}
Given $\tilde{R} = XY^T +YX^T$, with $Y^T u = 0$, we have
\begin{equation*}
\| \tilde{R} \|_1 \leq 5 \|\cR_{K_n,X}(Y)\|_1.
\end{equation*}
\end{pro}

Now we have in place all we need to prove lemma~~\ref{lem:bound on Rtilde}.
\begin{proof}[Proof (Lemma~\ref{lem:bound on Rtilde})]
Define the operator $\mathcal{A}_G : \mathbb{R}^{n \times n} \to \mathbb{R}^{E}$ as $\mathcal{A}_G(S) = [\langle M_{ij}, S\rangle]_{(i,j) \in E}$. By our assumptions,
\begin{align*}
&| \langle M_{ij},\tilde{R} \rangle + \langle M_{ij}, R^{\perp} \rangle |
= \left | \langle M_{ij},Q \rangle - \langle M_{ij}, Q_0 \rangle \right | \\ 
&\leq | \langle M_{ij},Q \rangle - \tilde{d}_{ij}^2 | +
\underbrace{ | \tilde{d}_{ij}^2 - \langle M_{ij}, Q_0 \rangle |} _{|z_{ij}|}
\leq 2 \Delta.
\end{align*}
Therefore,
$\| \mathcal{A}_G (\tilde{R})\|_1 \leq 2 |E| \Delta + \| \mathcal{A}_G (R^{\perp})\|_1$. Write the Laplacian matrix $\mathcal{L}$ as $\mathcal{L} = \sum_{(i,j) \in E} M_{ij}$. Then, $\langle \mathcal{L},R^{\perp} \rangle =  \sum_{(i,j)\in E} \langle M_{ij},R^{\perp} \rangle = \| \mathcal{A}_G (R^{\perp})\|_1$. Here, we used the fact that $ \langle M_{ij},R^{\perp} \rangle \geq 0$, since $M_{ij} \succeq \bf{0}$ and $R^{\perp} \succeq \bf{0}$. Hence, $ \| \mathcal{A}_G (\tilde{R})\|_1\leq
 2 |E| \Delta + \langle \mathcal{L},R^{\perp}\rangle $. 
 Due to Theorem~\ref{thm:OmegaL}, Eq.~\eqref{eq:ub_OmegaRperp}, and Claim~\ref{claim:sigma_max_omega}, 
\begin{equation*}
\langle \mathcal{L},R^{\perp}\rangle \leq C(nr^d)^3 r^{-2}\langle \Omega,R^{\perp}\rangle \leq C(nr^d)^6 \frac{n}{r^2}\Delta,
\end{equation*}
whence we obtain
\begin{equation*}
\| \mathcal{A}_G (\tilde{R})\|_1\leq C(nr^d)^6 \frac{n}{r^2} \Delta.
\end{equation*}
 
 \noindent The last step is to write $\|\mathcal{A}_G(\tilde{R})\|_1$ more explicitly. Notice that, 
 \begin{equation*}
 \| \mathcal{A}_G (\tilde{R})\|_1 = \sum_{(l,k) \in E} 
| \langle M_{lk}, XY^T + YX^T \rangle | =
2 \sum_{(l,k) \in E} | \langle x_l - x_k,y_l - y_k\rangle | = 2 \|\cR_{G,X}(Y)\|_1.
\end{equation*}

\noindent Invoking Theorem~\ref{thm:cheeger} and Proposition~\ref{pro:aux2}, we have
\begin{align*}
\|\tilde{R}\|_1 &\leq Cr^{-d-2} \|\cR_{G,X}(Y)\|_1 \\
&= C r^{-d-2} \|\mathcal{A}_G(\tilde{R})\|_1 \leq C(nr^d)^5 \frac{n^2}{r^4} \Delta.
\end{align*}
\end{proof}
%
\subsection{Proof of Theorem~\ref{thm:cheeger}}
We begin with some definitions and initial setup.

\begin{define}
The \emph{$d$-dimensional hypercube} $M_d$ is the simple graph whose vertices are the $d$-tuples with entries in $\{0,1\}$ and whose edges are the pairs of $d$-tuples that differ in exactly one position. Also, we use $M^{(2)}_d$ to denote the graph with the same set of vertices as $M_d$, whose edges are the pairs of $d$-tuples that differ in at most two positions. 
\end{define}

\begin{define}
An \emph{isomorphism} of graphs $G$ and $H$ is a bijection between the vertex sets of $G$ and $H$, say $\phi : V(G) \to V(H)$, such that any two vertices $u$ and $v$ of $G$ are adjacent in $G$ if and only if $\phi(u)$ and $\phi(v)$ are adjacent in $H$. The graphs $G$ and $H$ are called isomorphic, denoted by $G \simeq H$ if an isomorphism exists between $G$ and $H$.  
\end{define}

\smallskip
\noindent{\bf Chains and Force Flows.} A \emph{chain} $G_{ij}$ between nodes $i$ and $j$ is a sequence of subgraphs $H_1,H_2,\cdots,H_k$ of $G$, such that, $H_p \simeq M_d^{(2)}$ for $1 \le p \le k$, $H_p \cap H_{p+1} \simeq M^{(2)}_{d-1}$ for $1\le p \le k-1$ and $H_p \cap H_{p+2}$ is empty for $1 \le p \le k-2$. Further, $i$ (resp. $j$) is connected to all vertices in $V(H_1)\setminus V(H_2)$ (resp. $V(H_k)\setminus V (H_{k-1})$). See Fig.~\ref{fig:chain} for an illustration of a chain in case $d=2$.

A \emph{force flow} $\gamma$ is a collection of chains $\{G_{ij}\}_{1 \leq i \neq j \leq n}$ for all $n\choose 2$ node pairs. Let $\Gamma$ be the collection of all possible $\gamma$. Consider the probability distribution induced on $\Gamma$ by selecting the chains between all node pairs in the following manner. Chains are chosen independently for different node pairs. Consider a particular node pair $(i,j)$. Let $\ell = \|x_i -x_j\|$ and $a = (x_i -x_j) / \|x_i -x_j\|$. Define $\tilde{r} = \frac{3r}{4 \sqrt{2}}$, and choose nonnegative numbers $m \in \mathbb{Z}$ and $\eta \in \reals$, such that, $\ell = m \tilde{r} + \eta$ and $\eta < \tilde{r}$. Consider the following set of points on the line segment between $x_i$ and $x_j$.
 \begin{equation*}
 \xi_k = x_i + \frac{\eta}{2} + (k-1)\tilde{r}a, \quad
 \text{for}\; 1 \leq k \leq m+1.
 \end{equation*}
 Construct the sequence of hypercubes in direction of $a$, with centers at $(\xi _ k + \xi_{k+1})/2$, and side length $\tilde{r}$. (See Fig.~\ref{fig:chain_proof} for an illustration). Denote the set of vertices in this construction by $\{z_k\}$. 
Now, partition the space $[-0.5,0.5]^d$ into hypercubes (bins) of side length $\frac{r}{8\sqrt{d}}$. From the proof of Proposition~\ref{pro:sampling_lemma}, w.h.p., every bin contains at least one of the nodes $\{x_k\}_{k \in [n]}$. For every vertex $z_k$, choose a node $x_k$ uniformly at random among the nodes in the bin that contains $z_k$. Hence, $\|x_k -z_k\| \leq \frac{r}{8}$ and
 \begin{equation*}
 \|x_l - x_k\| \leq \|x_l -z_l\| + \|z_l - z_k\| + \|z_k - x_k\| \leq
 \frac{r}{4} + \|z_l -z_k\|, \quad \forall l,k.
 \end{equation*}

 \noindent By wiggling points $\{z_k\}$ to nodes $\{x_k\}$, we obtain a perturbation of the sequence of hypercubes, call it $G_{ij}$. It is easy to see that $G_{ij}$ is a chain between nodes $i$ and $j$.
 \begin{figure} [!t]
\centering
\includegraphics*[viewport = 155 300 600 410, width = 4in]{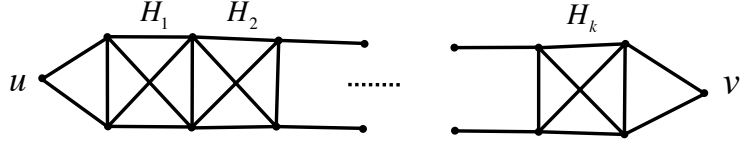}
\caption{\small An illustration of a chain $G_{uv}$} 
\label{fig:chain}
\end{figure}  
 
Under the above setup, we claim the following two lemmas.

\begin{lem}
\label{lem:congestion}
 Under the probability distribution on $\Gamma$ as described above, the expected number of chains containing a particular edge is upper bounded by $Cr^{-d-1}$, w.h.p., where $C = C(d)$ is a constant.
\end{lem}

The proof is discussed in Appendix~\ref{App:chain_exist}.
\begin{lem} \label{lem:cheeger_aux}
Let $G_{ij}$ be the chain between nodes $i$ and $j$ as described above. There exists a constant $C=C(d)$, such that,

\begin{equation*}
\left | \langle x_i -x_j, y_i - y_j \rangle \right | \leq Cr^{-1}\sum_{(l,k) \in E(G_{ij})} 
\left | \langle x_l -x_k, y_l - y_k \rangle \right |, \quad \forall 1\le i,j \le n.
\end{equation*}

\end{lem}
The proof is deferred to Section~\ref{sec:cheeger-aux}. Now, we are in position to prove Theorem~\ref{thm:cheeger}.

\begin{figure} [!t]
\centering
\includegraphics*[viewport =180 370 500 510, width = 5in]{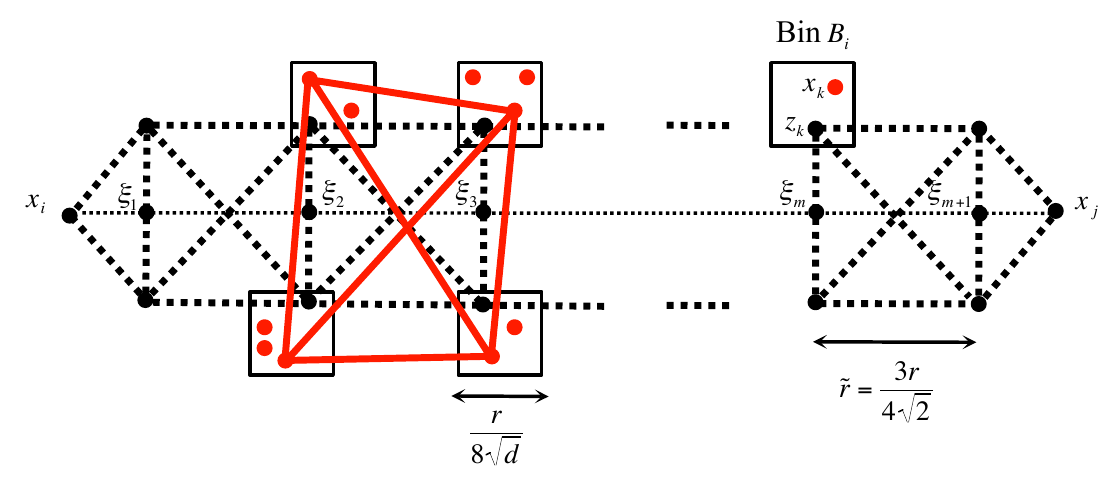}
\caption{\small Construction of chain $G_{ij}$ for case $d=2$.} 
\label{fig:chain_proof}
\end{figure}

%
%
%
%

%


\begin{proof}[Proof(Theorem~\ref{thm:cheeger})]
Consider a force flow $\gamma = \{G_{ij}\}_{1\le i,j \le n}$. Using lemma~\ref{lem:cheeger_aux}, we have

\begin{align}
\sum_{i,j} \left | \langle x_i -x_j, y_i - y_j \rangle \right | 
&\leq Cr^{-1} \sum_{i,j} \sum_{(l,k) \in E(G_{ij})}
\left | \langle x_l -x_k, y_l - y_k \rangle \right | \nonumber\\
&\leq Cr^{-1} \sum_{(l,k) \in E(G)} \Big( \sum_{G_{ij}: (l,k) \in E(G_{ij})} 1\Big) 
\left | \langle x_l -x_k, y_l - y_k \rangle \right | \nonumber\\
& = Cr^{-1} \sum_{(l,k) \in E(G)} b(\gamma,(l,k)) \left | \langle x_l -x_k, y_l - y_k \rangle \right | \label{eq:chain_congestion},
\end{align}
where $b(\gamma,(l,k))$ denotes the number of chains passing through edge $(l,k)$. Notice that in Eq.~\eqref{eq:chain_congestion}, $b(\gamma,(l,k))$ is the only term that depends on the force flow $\gamma$. Hence, $b(\gamma, (l,k))$ can be replaced by its expectation under a probability distribution on $\Gamma$. According to Lemma~\ref{lem:congestion}, under the described distribution on $\Gamma$, the average number of chains containing any particular edge is upper bounded by $Cr^{-d-1}$, w.h.p. Therefore,
\begin{equation*}
\sum_{i,j} \left | \langle x_i -x_j, y_i - y_j \rangle \right |
\leq Cr^{-d-2} \sum_{(l,k) \in E(G)} \left | \langle x_l -x_k, y_l - y_k \rangle \right |.
\end{equation*}
Equivalently, $\|\cR_{K_n,X}(Y)\|_1 \le Cr^{-d-2} \|\cR_{G,X}(Y)\|_1$, with high probability.
\end{proof}

\subsubsection{Proof of Lemma~\ref{lem:cheeger_aux}}
\label{sec:cheeger-aux}
\begin{proof}
Assume that $|V(G_{ij})|= m+1$ . Relabel the vertices in the chain $G_{ij}$ such that the nodes $i$ and $j$ have labels $0$ and $m$ respectively, and all the other nodes are labeled in $\{1,\cdots,m-1\}$. Since both sides of the desired inequality are invariant to translations, without loss of generality we assume that $x_0 = y_0 = 0$. For a fixed vector $y_m$ consider the following optimization problem:
\begin{equation*}
\Theta = \min_{y_1,\cdots,y_{m-1} \in \mathbb{R}^d} \sum_{(l,k) \in E(G_{ij})} 
\left | \langle x_l -x_k,y_l - y_k \rangle \right |.
\end{equation*}
To each edge $(l,k) \in E(G_{ij})$, assign a number $\lambda_{lk}$. (Note that $\lambda_{lk} = \lambda_{kl}$). For any assignment with $\max_{(l,k) \in E(G_{ij})} |\lambda_{lk}| \leq 1$, we have
\begin{align*}
\Theta &\geq  \min_{y_1,\cdots,y_{m-1} \in \mathbb{R}^d} \sum_{(l,k) \in E(G_{ij})} \lambda_{lk} \langle x_l - x_k,y_l - y_k\rangle\\ 
& =  \min_{y_1,\cdots,y_{m-1} \in \mathbb{R}^d} \sum_{\substack{ l \in G_{ij} \\l \neq 0}} \sum_{k \in \partial{l}} 
\lambda_{lk} \langle y_l,x_l - x_k \rangle\\
&=  \min_{y_1,\cdots,y_{m-1} \in \mathbb{R}^d} \sum_{\substack{ l \in G_{ij} \\l \neq 0}} \langle y_l,\sum_{k \in \partial{l}} \lambda_{lk} (x_l - x_k)\rangle,
\end{align*}
where $\partial{l}$ denotes the set of adjacent vertices to $l$ in $G_{ij}$. Therefore,
\begin{equation}
\label{eq:max_min}
\Theta \geq \max_{\lambda_{lk}: |\lambda_{lk}| \leq 1}\;
\min_{y_1,\cdots,y_{m-1} \in \mathbb{R}^d} \sum_{\substack{ l \in G_{ij} \\l \neq 0}} \langle y_l,\sum_{k \in \partial{l}} \lambda_{lk} (x_l - x_k)\rangle.
\end{equation}
Note that the numbers $\lambda_{lk}$ that maximize the right hand side should satisfy
$
\sum_{k \in \partial{l}} \lambda_{lk} (x_l - x_k) = 0, \forall l \neq 0, m.
$
Thus,
$
\Theta \geq \langle y_m,\sum_{k \in \partial{m}} \lambda_{mk} (x_m - x_k) \rangle$. Assume that we find values $\lambda_{lk}$ such that
\begin{equation}
\begin{split} \label{eq:constraints}
\begin{cases} 
\sum_{k \in \partial{l}} \lambda_{lk} (x_l - x_k) = 0 & \forall l \neq 0, m,\\
\sum_{k \in \partial{m}} \lambda_{mk} (x_m - x_k) = x_m,\\
\underset{(l,k) \in E(G_{ij})}{\max} |\lambda_{lk}| \leq Cr^{-1}.
\end{cases}
\end{split}
\end{equation}
Given these values $\lambda_{lk}$, define $\tilde{\lambda}_{lk} =\dfrac{\lambda_{lk}}{ \underset{(l,k) \in E(G_{ij})}{\max} |\lambda_{lk}|}$. Then $|\tilde{\lambda}_{lk}| \leq 1$ and
\begin{equation*}
\Theta \geq \langle y_m,\sum_{k \in \partial{m}} \tilde{\lambda}_{mk} (x_m - x_k) \rangle = 
\langle y_m, \frac{1}{\max_{l,k} |\lambda_{lk}| } x_m\rangle
\geq Cr \langle y_m, x_m\rangle,
\end{equation*}
which proves the thesis.

Notice that for any values $\lambda_{lk}$ satisfying~\eqref{eq:constraints}, we have
\begin{equation*}
\begin{cases}
&\sum_{l \in V(G_{ij})} \sum_{k \in \partial{l}} \lambda_{lk} (x_l - x_k) = \sum_{k \in \partial{0}} \lambda_{0k} (x_0 - x_k) + x_m\\
&\sum_{l \in V(G_{ij)}} \sum_{k \in \partial{l}} \lambda_{lk} (x_l - x_k) = \sum_{(l,k) \in E(G_{ij})} \lambda_{lk} (x_l - x_k) + \lambda_{kl} (x_k - x_l) = 0
\end{cases}
\end{equation*}
Hence, $\sum_{k \in \partial{0}} \lambda_{0k} (x_0 - x_k) = - x_m$. 

It is convenient to generalize the constraints in Eq.~\eqref{eq:constraints}. Consider the following linear system of equations with unknown variables $\lambda_{lk}$.
\begin{equation}
\sum_{k \in \partial{l}} \lambda_{lk}(x_l - x_k) = u_l, \quad \text{for} \; l = 0,\cdots,m.\label{eq:general_forces}
\end{equation}

\noindent Writing Eq.~\eqref{eq:general_forces} in terms of the rigidity matrix of $G_{ij}$, and using the characterization of its null space as discussed in section~\ref{sec:rigidity}, it follows that Eq.~\eqref{eq:general_forces} have a solution if and only if 
\begin{equation}
\label{eq:equilibrium}
\sum_{i=0}^m u_i =0,\quad
\sum_{i=0}^m u_i^T Ax_i = 0,
\end{equation}
where $A \in \mathbb{R}^{d \times d}$ is an arbitrary anti-symmetric matrix.

\bigskip
\noindent\textbf{A mechanical interpretation.} For any pair $(l,k) \in E(G_{ij})$, assume  a spring with spring constant $\lambda_{lk}$ between nodes $l$ and $k$. Then, by  Eq.~\eqref{eq:general_forces}, $u_l$ will be the force imposed on node $l$. The first constraint in Eq.~\eqref{eq:equilibrium} states that the net force on $G_{ij}$ is zero (\emph{force equilibrium}), while the second condition states that the net torque is zero (\emph{torque equilibrium}).

Indeed, $\sum_{i=0}^{m} u_i^T A u_i = \langle A, \sum_{i=0}^{m} u_i x_i^T\rangle = 0$, for every anti-symmetric matrix $A$ if and only if $\sum_{i=0}^{m} u_i x_i^T$ is a symmetric matrix. Therefore,
\begin{equation*}
\sum_{i=0}^{m} u_i \wedge x_i = \sum_{i=0}^{m}
\left(\sum_{\ell=1}^{d} u_i^{(\ell)} e_{\ell} \right) \wedge
\left(\sum_{k=1}^{d} x_i^{(k)} e_{k} \right)=
\sum_{\ell,k} \sum_{i=0}^{m} (u_i^{(\ell)} x_i^{(k)} - x_i^{(\ell)} u_i^{(k)}) (e_{\ell} \wedge e_k) = 0.
\end{equation*}

With this interpretation in mind, we propose a two-part procedure to find the spring constants $\lambda_{lk}$ that obey the constraints in~\eqref{eq:constraints}.
 
\bigskip
\textbf{Part (i):} For the sake of simplicity, we focus here on the special case $d=2$. The general argument proceeds along the same lines and is deferred to Appendix~\ref{App: stage_1_G}.

Consider the chain $G_{ij}$ between nodes $i$ and $j$, cf. Fig.~\ref{fig:chain}. For every $1 \le p \le k$, let $\mathcal{F}_p$ denote the common side of $H_p$ and $H_{p+1}$. Without loss of generality, assume $V(\mathcal{F}_p) = \{1,2\}$, and $x_m$ is in the direction of $e_1$. Find the forces $f_1$, $f_2$ such that
\begin{equation}
\begin{split} \label{eq:force_constraint}
f_1 + f_2 &= x_m, \quad f_1 \wedge x_1 + f_2 \wedge x_2 = 0,\\
&\|f_1\|^2+\|f_2\|^2 \leq C \|x_m\|^2.
\end{split}
\end{equation}
To this end, we solve the following optimization problem.
\begin{equation}
\begin{split} \label{eq:opt}
&\text{minimize} \quad \; \frac{1}{2} (\|f_1\|^2 + \|f_2\|^2)\\
&\text{subject to}  \quad f_1 + f_2 = x_m, \quad f_1 \wedge x_1 + f_2 \wedge x_2 = 0
\end{split}
\end{equation}
It is easy to see that the solutions of~\eqref{eq:opt} are given by
\begin{align*}
 &\begin{cases}
   f_1 = \frac{1}{2} x_m + \frac{1}{2} \gamma A (x_1 - x_2)\\
   f_2 = \frac{1}{2} x_m - \frac{1}{2} \gamma A (x_1-x_2)
  \end{cases}\\
  &\gamma = -\frac{1}{\|x_1 -x_2\|^2} x_m^T A(x_1 + x_2),\quad 
  A = \begin{pmatrix} 0 &-1 \\ 1 & 0\end{pmatrix}.
\end{align*}

Now, we should show that the forces $f_1$ and $f_2$ satisfy the constraint $\|f_1\|^2 + \|f_2\|^2 \leq C \|x_m\|^2$, for some constant $C$. Clearly, it suffices to prove $\|\gamma (x_1 - x_2)\| \leq C \|x_m\|$. Observe that
\begin{align*}
\frac{\|\gamma (x_1 - x_2)\|}{\|x_m\|} &= 
\frac{1}{\|x_1 -x_2\|} \bigg| \frac{x_m^T}{\|x_m\|} A(x_1 + x_2) \bigg|\\
&= \frac{1}{\|x_1 -x_2\|} |e_1^T A(x_1 + x_2)|\\
&= \frac{1}{\|x_1 -x_2\|} |e_2^T (x_1 + x_2)|.
\end{align*}
From the construction of chain $G_{ij}$, we have
\begin{equation*}
|e_2^T (x_1 + x_2)| \leq \frac{r}{4}, \quad \|x_1 - x_2\| \geq \frac{r}{4},
\end{equation*}
which shows that $\|\gamma (x_1 -x_2)\| \leq \|x_m\|$.
 
\bigskip
\textbf{Part (ii):} For each $H_p$ consider the following set of forces
 \begin{equation}
   u_{i} = 
   \begin{cases}
   f_i \; &\text{if} \; i \in V(\mathcal{F}_p)\\
   -f_i & \text{if} \; i \in V(\mathcal{F}_{p-1})
   \end{cases}, \; 
\end{equation}
Also, let $u_0 = -x_m$ and $u_m  = x_m$. (cf. Fig.~\ref{fig:force}). 

Notice that $\sum_{i \in V(H_p)} u_i = 0$, $ \sum_{i \in V(H_p)} u_i \wedge x_i = 0
$, and thus by the discussion prior to Eq.~\eqref{eq:equilibrium}, there exist values $\lambda_{lk}^{(H_p)}$, such that
\begin{equation*}
\sum_{k:(l,k) \in E(H_p)} \lambda_{lk}^{(H_p)} (x_l - x_k) = u_l, \quad \forall l \in V(H_p).
\end{equation*}
Writing this in terms of $R^{(H_p)}$, the rigidity matrix of $H_p$, we have 
\begin{equation}
\label{eq:Rig_lambda}
(R^{(H_p)})^T \underline{\lambda}^{(H_p)} = \underline{u},
\end{equation}
where the vector $\underline{\lambda}^{(H_p)} = [\lambda^{(H_p)}_{lk}]$ has size $|E(H_p)| = d(d+1) 2^{d-2}$, and the vector $\underline{u} = [u_l]$ has size $d \times |V(H_p)| = d 2^d$.
Among the solutions of Eq.~\eqref{eq:Rig_lambda}, choose the one that is orthogonal to the nullspace of $(R^{(H_p)})^T$.
Therefore, 
\begin{equation*}
\sigma_{\min}(R^{(H_p)}) \|\lambda^{(H_p)}\|_{\infty} \leq \sigma_{\min}(R^{(H_p)}) \|\lambda^{(H_p)}\|_2 \leq \|u\| \leq C \|x_m\|.
\end{equation*}

Form the construction of the chains, $H_p$ is a perturbation of the d-dimensional hypercube with side length $\tilde{r} = \frac{3r}{4\sqrt{2}}$. (each vertex wiggles by at most $\frac{r}{8}$). Using the fact that $\sigma_{\min}(.)$ is a Lipschitz continuos function of its argument, we get that $\sigma_{\min}(R^{(H_p)}) \geq Cr$, for some constant $C =C(d)$. Also, $\|x_m\| \leq 1$. Hence, $\|\lambda^{(H_p)}\|_\infty \leq Cr^{-1}$.
 
\begin{figure} [!t]
\centering
\includegraphics*[viewport = 170 300 420 410, width = 3in]{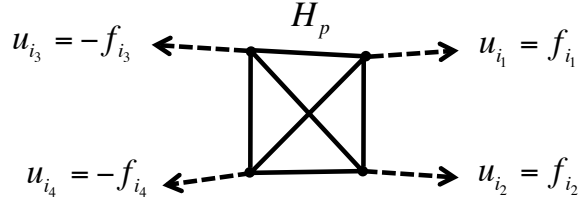}
\caption{\small $H_p$ and the set of forces in Part $(ii)$} 
\label{fig:force}
\end{figure}   

Now define
\begin{equation}
\label{eq:def_lambda}
\lambda_{lk} = \sum_{H_p : (l,k) \in E(H_p)} \lambda^{(H_p)}_{lk},
\quad \quad \forall (l,k) \in E(G_{ij}).
\end{equation}
We claim that the values $\lambda_{lk}$ satisfy the constraints in~\eqref{eq:constraints}. 

First, note that for every node $l$,
\begin{align}
\sum_{k \in \partial{l}} \lambda_{lk} (x_l - x_k) &=
\sum_{k \in \partial{l}} \left( \sum_{H_p : (l,k) \in E(H_p)} \lambda_{lk}^{(H_p)} \right) (x_l -x_k) \nonumber\\
&= \sum_{H_p : l \in V(H_p)} \left( \sum_{k: (l,k) \in E(H_p)} \lambda_{lk}^{(H_p)} (x_l - x_k) \right) \nonumber\\
&= \sum_{H_p : l \in V(H_p)} u_l. \label{eq:sum_f}
\end{align}
For nodes $l \notin \{0,m\}$, there are two $H_p$ containing $l$. In one of them, $u_l = f_l$ and in the other $u_l = -f_l$. Hence, the forces $u_l$ cancel each other in Eq.~\eqref{eq:sum_f} and the sum is zero. At nodes $0$ and $m$, this sum is equal to $-x_m$ and $x_m$ respectively.

Second, since each edge participates in at most two $H_p$, it follows from Eq.~\eqref{eq:def_lambda} that $|\lambda_{lk}| \leq C r^{-1}$.
\end{proof}

\section{Proof of Theorem~\ref{thm:main_result} (Lower Bound)}
\label{sec:converse}
\begin{proof}{
Consider the `\emph{bending}' map $\mathcal{T}: [-0.5,0.5]^d \to \mathbb{R}^{d+1}$, defined as
\begin{equation*}
 \mathcal{T}(t_1,t_2,\cdots,t_d) = (R\sin \frac{t_1}{R},t_2,\cdots, t_d,R(1-\cos \frac{t_1}{R}))
 \end{equation*}
This map bends the hypercube in the $d+1$ dimensional space. Here, $R$ is the curvature radius of the embedding (for instance, $R\gg1$ corresponds to slightly bending the hypercube, cf. Fig.~\ref{fig:Bending_map}).

 \begin{figure} [!t]
\centering
\includegraphics*[viewport = 150 220 600 380, width = 5in]{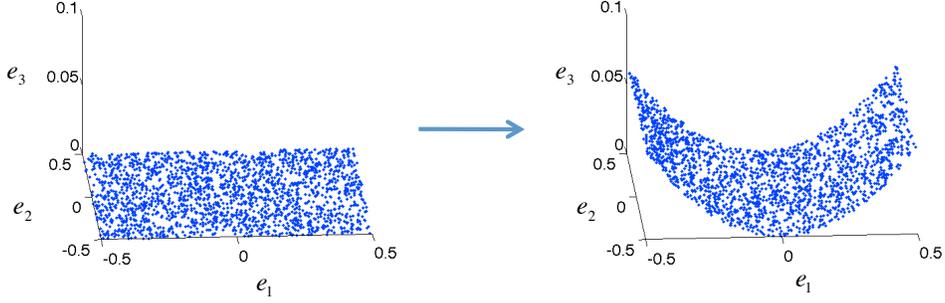}
\caption{\small Bending map $\mathcal{T}$, with $R =2$, and $d=2$.} 
\label{fig:Bending_map}
\end{figure} 

Now for a given $\Delta$, let $R = \max\{1,r^2 \Delta^{-1/2}\}$ and give the distances $\tilde{d}_{ij} = \|\mathcal{T}(x_i) - \mathcal{T}(x_j)\|$ as the input distance measurements to the algorithm.
First we show that these adversarial measurements satisfy the noise constraint $\|\tilde{d}^2_{ij} - d^2_{ij}\| \leq \Delta$.

\begin{align*}
d^2_{ij} - \tilde{d}^2_{ij}  &= (x_{i}^{(1)} - x_{j}^{(1)})^2 - R^2 \Big[ \sin \Big(\frac{x_{i}^{(1)}}  {R}\Big) - \sin \Big(\frac{x_{j}^{(1)}}  {R} \Big) \Big]^2 \\
& -R^2 \Big[ \cos \Big(\frac{x_{i}^{(1)}} {R}\Big) - \cos \Big(\frac{x_{j}^{(1)}} {R}\Big) \Big]^2 \\
&=  (x_{i}^{(1)} - x_{j}^{(1)})^2 - R^2 \Big[ 2 - 2 \cos \Big( \frac{x_{i}^{(1)} - x_{j}^{(1)}}  {R} \Big) \Big] \\
&\leq \frac{( x_{i}^{(1)} - x_{j}^{(1)} )^4} {2R^2} \leq \frac{r^4}{2 R^2} \le \Delta.
\end{align*}
Also, $\tilde{d}_{ij} \leq d_{ij} $. Therefore, $|z_{ij}| = |\tilde{d}^2_{ij} - d^2_{ij}| \leq \Delta$. 

The crucial point is that the SDP in the first step of the algorithm is oblivious of dimension $d$. Therefore, given the measurements $\tilde{d}_{ij}$ as the input, the  SDP will return the Gram matrix $Q$ of the positions $ \tilde{x}_i =  L \mathcal{T}(x_i)$, i.e., $Q_{ij} = \tilde{x}_i \cdot \tilde{x}_j$.  Denote by $\{u_1,\cdots,u_d\}$, the eigenvectors of $Q$ corresponding to the $d$ largest eigenvalues. Next, the algorithm  projects the positions $\{\tilde{x}_i\}_{i \in [n]}$ onto the space $U = \langle u_1,\cdots, u_d\rangle$ and returns them as the estimated positions in $\mathbb{R}^d$. 
Hence,
\begin{equation*}
d(X,\hX) = \frac{1}{n^2} \|XX^T - P_U \tilde{X}\tilde{X}^T P_U\|_1.
\end{equation*}

Let $W = \langle e_1,e_2,\cdots,e_d \rangle$ (see Fig.~\ref{fig:Planes}). Then,
\begin{equation}
\label{eq:converse_bound}
d(X,\hX) \geq \frac{1}{n^2} \|XX^T - \tilde{X} P_W \tilde{X}^T\|_1- \frac{1}{n^2}\|\tilde{X} P_W \tilde{X}^T - P_U \tilde{X}\tilde{X}^T P_U\|_1.                     
\end{equation}

We bound each terms on the right hand side separately. For the first term,
\begin{align}
&\frac{1}{n^2} \|XX^T - \tilde{X} P_W \tilde{X}^T \|_1\nonumber\\
&= \frac{1}{n^2} \sum_{1 \leq i,j \leq n} \bigg|x_i^{(1)}x_j^{(1)} - R^2 \sin \Big(\frac{x_i^{(1)}}{R}\Big) \sin \Big(\frac{x_j^{(1)}}{R}\Big) \bigg| \nonumber\\
& \stackrel{(a)}{=} \frac{1}{n^2} \sum_{1 \leq i,j \leq n}
\bigg|x_i^{(1)} x_j^{(1)} - R^2 \Big(\frac{x_i^{(1)}}{R} - \frac{(x_i^{(1)})^3}{3! R^3} +
 \frac{\xi_i^5}{5!R^5}\Big) \Big(\frac{x_j^{(1)}}{R} - \frac{(x_j^{(1)})^3}{3! R^3} +
 \frac{\xi_j^5}{5! R^5} \Big)\bigg| \nonumber\\
& \stackrel{(b)}{\geq} C\bigg(\frac{R}{n} \bigg)^2 \sum_{1 \leq i,j \leq n} \bigg|\frac{1}{3!} \Big(\frac{x_i^{(1)}}{R}\Big)\Big(\frac{x_j^{(1)}} {R}\Big)^3 + \frac{1}{3!}\Big(\frac{x_j^{(1)}}{R}\Big) \Big(\frac{x_i^{(1)}} {R} \Big)^3\bigg| \nonumber\\
&\geq \frac{C}{(nR)^2} \Big(\sum_{ 1\leq i \leq n} |x_i^{(1)}| \Big)\Big(\sum_{ 1\leq j \leq n} |x_i^{(1)}|^3 \Big) \geq \frac{C}{R^2}
\label{eq:bound_term1},
\end{align}
where $(a)$ follows from Taylor's theorem, and $(b)$ follows from $|\xi_i / R| \leq |x_i / R| \leq 1/2$. 
 
\begin{figure} [!t]
\centering
\includegraphics*[viewport = 40 50 720 510, width = 3in]{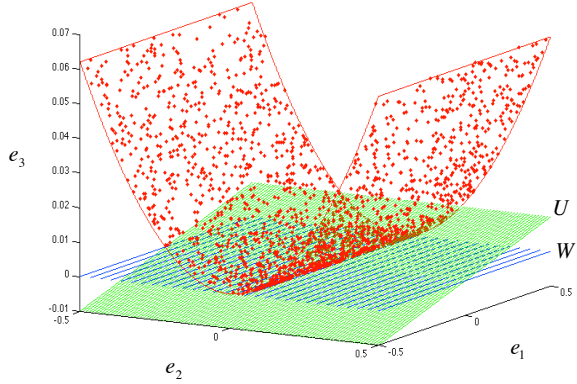}
\caption{\small An illustration of subspaces $U$ and $W$.} 
\label{fig:Planes}
\end{figure} 
 
 The next Proposition provides an upper bound for the second term on the right hand side of Eq.~\eqref{eq:converse_bound}.
 
 \begin{pro}
 \label{pro:bound_projections}
The following is true.
\begin{equation*}
 \frac{1}{n^2}\|\tilde{X} P_W \tilde{X}^T - P_U \tilde{X}\tilde{X}^T P_U\|_1 \to 0 \text{ a.s.}, \quad \text{as } n\to \infty.
\end{equation*}
 \end{pro}
 
 Proof of this Proposition is provided in the next section.
 
Using the bounds given by Proposition~\ref{pro:bound_projections} and Eq.~\eqref{eq:bound_term1},
 we obtain that, w.h.p.,
 \begin{equation*}
 d(X,\hX) \geq \frac{C_1}{R^2}
               \geq C \min\{1,\frac{\Delta}{r^4}\}.
 \end{equation*}
 The result follows.
}\end{proof}
 %
\subsection{Proof of Proposition~\ref{pro:bound_projections}}

We first establish the following remarks.
\begin{remark}
\label{rem:chord_dist}
Let $a,\; b\in \reals^{m}$ be two unitary vectors. Then,
\begin{equation*}
\|aa^T - bb^T\|_2 = \sqrt{1- (a^Tb)^2}.
\end{equation*}
\end{remark}
For proof, we refer to Appendix~\ref{App:chord_dist}

\begin{remark}
\label{rem:Weyl}
Assume $A$ and $\tilde{A}$ are $p \times p$ matrices. Let $\{\lambda_i\}$ be the eigenvalues of A such that $\lambda_1 \geq \cdots \geq \lambda_{p-1} > \lambda_p$. Also, let $v$ and $\tilde{v}$ respectively denote the eigenvectors of $A$ and $\tilde{A}$ corresponding to their smallest eigenvalues. Then,

\begin{equation*}
1- (v^T \tilde{v})^2 \leq \frac{4\|A - \tilde{A}\|_2}{\lambda_{p-1} - \lambda_p}.
\end{equation*}
\end{remark}  

The proof is deferred to Appendix~\ref{App:Weyl}.

 \begin{proof}[Proof(Proposition~\ref{pro:bound_projections})]
 Let $\tilde{X} = \sum_{i=1}^{d+1} \sigma_i u_i \hat{w}_i^T$ be the singular value decomposition of $\tX$, where $\|u_i\| = \|\hat{w}_i\|=1$, $u_i \in \reals^{n}$, $\hat{w}_i \in \reals^{d+1}$ and $\sigma_1 \geq \sigma_2 \geq \cdots \geq \sigma_{d+1}$. Notice that 
 \begin{equation*}
 P_U \tilde{X} = \sum_{i=1}^d \sigma_i u_i \hat{w}_i^T = 
 (\sum_{i=1}^{d+1} \sigma_i u_i \hat{w}_i^T) (\sum_{j=1}^d \hat{w}_j  \hat{w}_j^T)=
\tX P_{\hW},
 \end{equation*}
where $\hW = \langle \hat{w}_1, \cdots, \hat{w}_d \rangle$, and $P_{\hW} \in \reals^{(d+1) \times (d+1)}$. Hence, $P_U \tX {\tX}^T P_U = \tX P_{\hW} {\tX}^T$. Define $M = P_{\hW} - P_W$. Then, we have
\begin{align}
\frac{1}{n^2} \|\tX P_W \tX^T - P_U \tX \tX^T P_U\|_1 &=
\frac{1}{n^2} \|\tX M \tX^T\|_1 = 
\frac{1}{n^2} \sum_{1 \leq i,j \leq n} |\tilde{x}_i^T M \tilde{x}_j| \nonumber\\
&\leq \frac{1}{n^2} \|M\|_2 \sum_{1 \leq i,j \leq n} \|\tilde{x}_i\| \|\tilde{x}_j\|
\leq \|M\|_2. \label{eq:bound by M_op}
\end{align}

Now, we need to bound $\|M\|_2$.  We have,
\begin{equation*}
M = P_{\hW} -P_W = (I - P_{\hat{w}_{d+1}}) - (I - P_{e_{d+1}}) = P_{e_{d+1}} - P_{\hat{w}_{d+1}}.
\end{equation*}
Using Remark~\ref{rem:chord_dist}, we obtain $\|M\|_2 = \|e_{d+1}^{}e_{d+1}^T - \hat{w}_{d+1}^{}\hat{w}_{d+1}^T\|_2 = \sqrt{1 - (e_{d+1}^T \hat{w}_{d+1}^{})^2}$.

Let $Z_i = \tilde{x}_i \tilde{x}_i^T \in \reals^{(d+1) \times (d+1)}$, $\bar{Z} = \frac{1}{n} \sum_{i=1}^{n} Z_i$ and $Z = \mathbb{E}(Z_i)$. Notice that $\bar{Z} = \frac{1}{n} \tX^T \tX = \frac{1}{n} \sum_{i=1}^{d+1} \sigma_i^2 \hat{w}_i \hat{w}_i^T$. Therefore, $\hat{w}_{d+1}$ is the eigenvector of $\bar{Z}$ corresponding to its smallest eigenvalue. In addition, $Z$ is a diagonal matrix (with $Z_{(d+1),(d+1)}$ the smallest diagonal entry). Hence, $e_{d+1}$ is its eigenvector corresponding to the smallest eigenvalue, $Z_{(d+1),(d+1)}$.

By applying Remark~\ref{rem:Weyl}, we have
\begin{equation}
\label{eq:bound on M_op}
\|M\|_2 \leq  \sqrt{1 - (e_{d+1}^T \hat{w}_{d+1}^{})^2} \leq  \sqrt{\frac{4\|Z - \bar{Z}\|_2}{\lambda_d - \lambda_{d+1}}}, 
\end{equation}
where $\lambda_d >\lambda_{d+1}$ are the two smallest eigenvalues of $Z$.  Let $t$ be a random variable, uniformly distributed in $[-0.5,0.5]$. Then, 
\begin{equation*}
\lambda_{d} = \mathbb{E}\bigg[R^2 \sin^2 \bigg(\frac{t}{R}\bigg) \bigg] \quad \text{and} \quad 
\lambda_{d+1} = \mathbb{E} \bigg[R^2 \bigg(1- \cos^2 \bigg(\frac{t}{R} \bigg) \bigg) \bigg].
\end{equation*}
 Hence, $\lambda_d -\lambda_{d+1} = R^3 (-1/R -\sin(1/R) + 4 \sin(1/2R)) \geq 0.07$, since $R \geq 1$. 

Also, note that $\{Z_i\}_{1\le i \le n }$ is a sequence of iid random matrices with dimension $(d+1)$ and $\|Z\|_{\infty} = \|\E(Z_i)\|_{\infty} < \infty$. By Law of large numbers, $\bar{Z} \to Z$, almost surely. Now, since the operator norm is  a continuos function, we have $\|Z - \bar{Z}\|_2 \to 0$, almost surely. The result follows directly from Eqs.~\eqref{eq:bound by M_op} and~\eqref{eq:bound on M_op}. 


 \end{proof}
\section{Numerical experiments}
\label{sec:simulation}
Theorem~\ref{thm:main_result} considers a worst case model for the measurement noise in which the errors $\{z_{ij}\}_{(i,j) \in E}$ are arbitrary but uniformly bounded as $|z_{ij}| \le \Delta$. The proof of the lower bound (cf. Section~\ref{sec:converse}) introduces errors $\{z_{ij}\}_{(i,j)\in E}$ defined based on a bending map, $\mathcal{T}$. This set of errors results in the claimed lower bound. For clarity, we denote this set of errors by $\{z^{\mathcal{T}}_{ij}\}$. In this section, we consider a mixture model for the measurement errors. For given parameters $\Delta$ and $\ve$, we let 
\begin{eqnarray}\label{eqn:noise_sim}
z_{ij} \sim \ve \gamma_{\Delta/2} + (1-\ve) \delta_{z^{\mathcal{T}}_{i,j}},
\end{eqnarray}
where $\gamma_{\sigma}(x) = 1/(\sqrt{2\pi} \sigma) e^{-x^2/2\sigma^2}$ is the density function of the normal distribution with mean zero and variance $\sigma^2$. The goal of the numerical experiments is to show the dependency of the algorithm performance on each of the parameters $n, r$ and $\Delta$. We consider the following configurations. For each configuration we run the SDP-based algorithm and evaluate $d(X,\hX)$. The error bars in figures correspond to $10$ realizations of that configuration. Throughout the measurement errors are defined according to~\eqref{eqn:noise_sim} with $\ve = 0.1$. 
\begin{enumerate}
\item Fix $\Delta = 0.005$ and $d \in \{2,4\}$. Let $r = 3 (\log n/n)^{1/d}$, with $n \in \{100,120, 140,\cdots, 300\}$. Fig.~\ref{fig:sweep_n} summarizes the results. According to the plot, $d(X,\hX) \propto n^2$ for $d = 2$ and $d(X,\hX) \propto n$ for $d = 4$.   

\item Fix $\Delta = 0.005$, $d = 2$ and $n = 150$. Let $r\in\{0.5, 0.55, 0.6, \cdots, 0.8\}$. The results are shown in Fig.~\ref{fig:sweep_r}. As we see, $d(X,\hX)$ is fairly proportional to $r^{-4}$.

\item Fix $n = 150$, $r = 0.6$ and $d = 2$. Let $\Delta \in \{0.005, 0.01, 0.015, 0.02, 0.025\}$. Fig.~\ref{fig:sweep_delta} showcases the results. The performance deteriorates linearly with respect to $\Delta$.
\end{enumerate}

\textbf{Acknowledgment.} Adel Javanmard is supported by Caroline and Fabian Pease Stanford Graduate Fellowship. This work was partially supported by the NSF CAREER award CCF-0743978, the NSF grant DMS-0806211, and the AFOSR grant FA9550-10-1-0360. The authors thank the anonymous reviewers for their insightful comments.

\begin{figure}[!t]
\centering
\includegraphics*[viewport = 30 175 540 600, width = 3.4in]{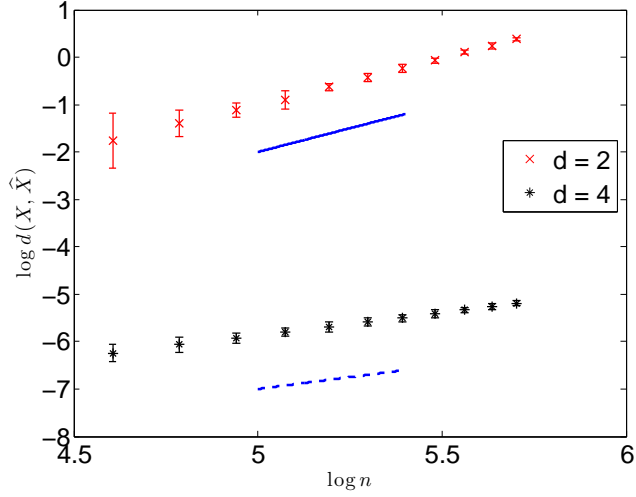}
\caption{\small Performance results for $\Delta = 0.005$, $d= 2, 4$, and $r = 3(\log n/n)^{1/d}$. The plot shows $\log\, d(X,\hX)$ vs. $\log n$ for a set of values of $n$. The solid line and the dashed line respectively correspond to $d(X,\hX) \propto n^2$ and $d(X,\hX) \propto n$ and are plotted as reference.} 
\label{fig:sweep_n}
\end{figure}  

\vspace{4cm}

\begin{figure}[!h]
\centering
\includegraphics*[viewport = 30 175 545 600, width = 3.5in]{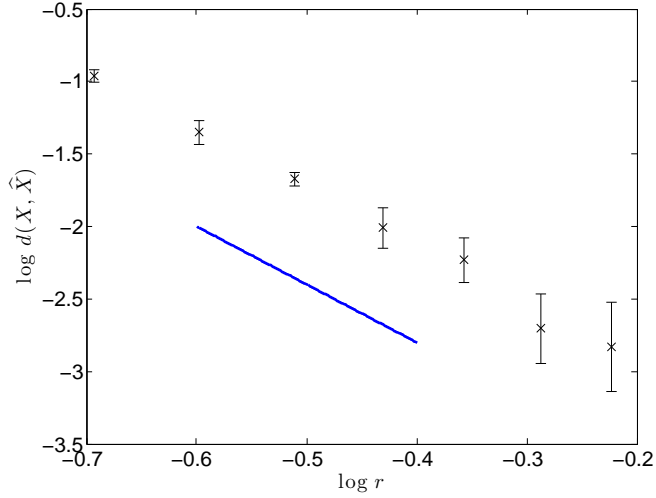}
\caption{\small Performance results for $\Delta = 0.005$, $d=2$, and $n = 150$. The plot shows $\log\, d(X,\hX)$ vs. $\log r$ for a set of values of $r$. The solid line corresponds to $d(X,\hX) \propto r^{-4}$ and is plotted as reference.} 
\label{fig:sweep_r}
\end{figure}  

\begin{figure}[!h]
\centering
\includegraphics*[viewport = 30 175 545 650, width = 3.5in]{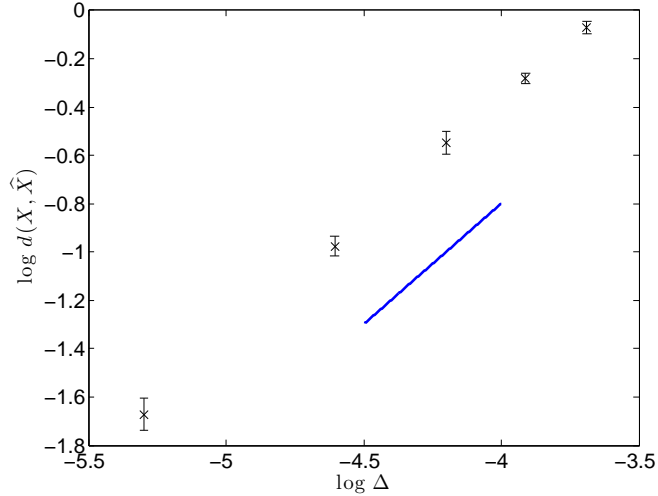}
\caption{\small Performance results for $n = 150$, $r = 0.6$, and $d = 2$. The plot shows $\log\, d(X,\hX)$ vs. $\log \Delta$ for a set of values of $\Delta$. The solid line corresponds to $d(X,\hX) \propto \Delta$ and is plotted as reference.} 
\label{fig:sweep_delta}
\end{figure} 
 
\newpage
\appendix
\section{Proof of Remark~\ref{rem:region}}
\label{App:region}
For $1\le j \le n$, let random variable $z_j$ be $1$ if node $j$ is in region $\mathcal{R}$ and 0 otherwise. The variables $\{z_j\}$ are i.i.d. Bernoulli with probability $V(\mathcal{R})$ of success. Also, $n(\mathcal{R}) = \sum_{j =1}^n z_j$. 

\noindent By application of the Chernoff bound we obtain
\begin{equation*}
\P \Big(\Big| \sum_{j =1}^n z_j - n V(\mathcal{R}) \Big| \geq \delta nV(\mathcal{R}) \Big) \geq 2 \exp 
\Big( - \frac{\delta^2 n V(\mathcal{R})}{2} \Big).
\end{equation*}
\noindent Choosing $\delta = \sqrt{\dfrac{2c\log n}{n V(\mathcal{R})}}$, the right hand side becomes $2 \exp(-c \log n) = 2/n^c$. Therefore, with probability at least $1- 2/n^c$,
\begin{equation}
\label{eq:n(R)}
n(\mathcal{R}) \in nV(\mathcal{R}) + [- \sqrt{2c n V(\mathcal{R})\log n},  \sqrt{2c n V(\mathcal{R}) \log n}].
\end{equation}

\section{Proof of Proposition~\ref{pro:sampling_lemma}}
\label{App:sampling_lemma}

We apply the bin-covering technique. Cover the space $[-0.5,0.5]^d$ with a set of non-overlapping hypercubes (bins) whose side lengths are $\delta$. Thus, there are a total of $m= \lceil{1/\delta}\rceil^d$ bins, each of volume $\delta^d$. In formula, bin $(j_1,\cdots, j_d)$ is the hypercube $[(j_1-1) \delta, j_1 \delta) \times \cdots \times [(j_d-1) \delta, j_d \delta)$, for $j_k \in \{1,\cdots,\lceil1 /\delta\rceil\}$ and $k \in \{1, \cdots, d\}$. Denote the set of bins by $\{B_k\}_{1 \le k \le m}$. Assume $n$ nodes are deployed uniformly at random in $[-0.5,0.5]^d$. We claim that if $\delta \ge  (c \log n /n)^{1/d}$, where $c >1$, then w.h.p., every bin contains at least $d+1$ nodes.  

Fix $k$ and let random variable $\xi_l$ be $1$ if node $l$ is in bin $B_k$ and $0$ otherwise. The variables $\{\xi_l\}_{1\le l \le n}$ are i.i.d. Bernoulli with probability $1/m$ of success. Also $\xi = \sum_{l=1}^n \xi_l$ is the number of nodes in bin $B_k$. By Markov inequality, $\P(\xi \le d) \le \E \{Z^{\xi -d}\}$, for any $0 \le Z \le 1$. Choosing $Z = md/n$, we have 
\begin{align*}
\P (\xi \le d) &\le \E\{Z^{\xi -d}\} = Z^{-d} \prod_{l=1}^{n} \E\{Z^{\xi_l}\}\\
 & = Z^{-d} \left(\frac{1}{m} Z + 1 - \frac{1}{m} \right)^n 
 = \left(\frac{n}{md} \right)^d \left(1 + \frac{d}{n} - \frac{1}{m} \right)^n\\
 & \le \left( \frac{ne}{md} \right)^d e^{-n/m} = \left(\frac{ne\delta^d}{d} \right)^d e^{-n \delta^d}
 \le \Big(\frac{c e\log n}{d}\Big)^d n^{-c}.
\end{align*} 
By applying union bound over all the $m$ bins, we get the desired result.

Now take $\delta = r/(4\sqrt{d})$. Given that $r \ge 4c \sqrt{d} (\log n / n)^{1/d}$, for some $c >1$, every bin contains at least $d+1$ nodes, with high probability.
Note that for any two nodes $x_i,x_j \in[-0.5,0.5]^d$ with $\|x_i - x_j\| \le r/2$, the point $(x_i+x_j)/2$ (the midpoint of the line  segment between $x_i$ and $x_j$) is contained in one of the bins, say $B_k$. For any point $s$ in this bin,

\begin{equation*}
\|s - x_i\| \leq \Big\|s - \frac{x_i + x_j}{2} \Big\|+ \Big\|\frac{x_i + x_j}{2} - x_i \Big\| \leq
\frac{r}{4} + \frac{r}{4} = \frac{r}{2}.
\end{equation*}

Similarly, $\|s - x_j\| \leq r/2$. Since $s \in B_k$ was arbitrary, $\Ci \cap \Cj$ contains all the nodes in $B_k$. This implies the thesis, since $B_k$ contains at least $d+1$ nodes.
\section{Proof of Proposition~\ref{pro:stress_construction}}
\label{App:stress_matrix}
Let $m_k = |\Q_k|$ and define the matrix $R_k$ as follows.
\begin{equation*}
R_k = \begin{bmatrix} 
x_{\Q_k}^{(1)}\Big| \cdots| x_{\Q_k}^{(d)} \Big| u_{\Q_k}
\end{bmatrix}^T \in \reals^{(d+1) \times m_k}.
\end{equation*} 
Compute an orthonormal basis $w_{k,1},\cdots,w_{k,m_k - d-1} \in \reals^{m_k}$ for the nullspace of $R_k$. Then 
\begin{equation*}
\Omega_k = P^{\perp}_{ \langle u^{}_{\Q_k},x^{(1)}_{\Q_k},\cdots,x^{(d)}_{\Q_k}\rangle} = \sum_{l=1}^{m_k -d-1} w_{k,l} w_{k,l}^T.
\end{equation*}
Let $\hat{w}_{k,l} \in \reals^{n}$ be the vector obtained from $w_{k,l}$ by padding it with zeros.  Then, $\hat{\Omega}_k = \sum_{l=1}^{m_k -d-1} \hat{w}_{k,l} \hat{w}_{k,l}^T$. In addition, the $(i,j)$ entry of $\hat{\Omega}_k$ is nonzero only if $i,j \in \Q_k$. Any two nodes in $\Q_k$ are connected in $G$ (Recall that $\Q_k$ is a cliques of $G$). Hence, $\hat{\Omega}_k$ is zero outside $E$. Since $\Omega = \sum_{\Q_k \in \cliq(G)} \hat{\Omega}_k$, the matrix $\Omega$ is also zero outside $E$.

Notice that for any $v \in \langle x^{(1)},\cdots,x^{(d)}, u\rangle$,
\begin{equation*}
\Omega v = (\sum_{\Q_k \in \cliq(G)}  \hat{\Omega}_k) v =\sum_{\Q_k \in \cliq(G)}  \Omega_k v^{}_{\Q_k} = 0.
\end{equation*}

So far we have proved that $\Omega$ is a stress matrix for the framework. Clearly, $\Omega \succeq 0 $, since $\hat{\Omega}_k \succeq 0$ for all $k$. We only need to show that $\rank(\Omega) = n-d-1$. Since $\text{Ker}(\Omega) \supseteq \langle x^{(1)},\cdots,x^{(d)}, u\rangle$, we have $\text{rank}(\Omega) \le n-d-1$. Define
\begin{equation*}
\tilde{\Omega} = \sum_{\Q_k \in \{\C_1,\cdots,\C_n\}} \hat{\Omega}_k.
\end{equation*}

\noindent Since $\Omega \succeq \tilde{\Omega} \succeq 0$, it suffices to show that $\rank(\tilde{\Omega}) \ge n -d -1$. For an arbitrary vector $v \in \text{Ker}(\tilde{\Omega})$,
\begin{equation*}
v^T \tilde{\Omega} v = \sum_{i=1}^n 
\| P^{\perp}_{\langle u^{}_{\C_i},x^{(1)}_{\C_i},\cdots,x^{(d)}_{\C_i}\rangle} v^{}_{\C_i}\|^2 = 0,
\end{equation*}
which implies that $v^{}_{\C_i} \in \langle u^{}_{\C_i},x^{(1)}_{\C_i},\cdots,x^{(d)}_{\C_i} \rangle$. Hence, the vector $v^{}_{\C_i}$ can be written as

\begin{equation*}
v^{}_{\mathcal{C}_i} = \sum_{\ell=1}^d \beta^{(\ell)}_{i} x^{(\ell)}_{\mathcal{C}_i}
+ \beta^{(d+1)}_{i} u^{}_{\mathcal{C}_{i}}
\end{equation*}
for some scalars $\beta^{(\ell)}_i$. Note that for any two nodes $i$ and $j$, the vector $v^{}_{\mathcal{C}_i \cap \mathcal{C}_j}$ has the following two representations

\begin{equation*}
v^{}_{\mathcal{C}_i \cap \mathcal{C}_j} = 
\sum_{\ell=1}^d \beta^{(\ell)}_{i} x^{(\ell)}_{\mathcal{C}_i \cap \mathcal{C}_j}
+ \beta^{(d+1)}_{i} u^{}_{\mathcal{C}_{i} \cap \mathcal{C}_j}=
\sum_{\ell=1}^d \beta^{(\ell)}_{j} x^{(\ell)}_{\mathcal{C}_i \cap \mathcal{C}_j}
+ \beta^{(d+1)}_{j} u^{}_{\mathcal{C}_{i} \cap \mathcal{C}_j}
\end{equation*} 
Therefore,
\begin{equation}
\label{eq:represent}
\sum_{\ell=1}^d (\beta^{(\ell)}_{i} - \beta^{(\ell)}_{j}) x^{(\ell)}_{\mathcal{C}_i \cap \mathcal{C}_j}
+ (\beta^{(d+1)}_{i} - \beta^{(d+1)}_{j}) u^{}_{\mathcal{C}_{i} \cap \mathcal{C}_j} = 0
\end{equation}
According to Proposition~\ref{pro:sampling_lemma}, with high probability, for any two nodes $i$ and $j$ with $\|x_i - x_j\| \leq r/2$, we have $|\mathcal{C}_i \cap \mathcal{C}_j| \geq d+1$. Thus, the vectors $x^{(\ell)}_{\mathcal{C}_i \cap \mathcal{C}_j}$, $u^{}_{\mathcal{C}_i \cap \mathcal{C}_j}$, $1 \leq \ell \leq d$ are linearly independent, since the configuration is generic. More specifically, let $Y$ be the matrix with $d+1$ columns $\{x^{(\ell)}_{\mathcal{C}_i \cap \mathcal{C}_j}\}_{\ell=1}^d$, $u^{}_{\mathcal{C}_i \cap \mathcal{C}_j}$. Then, ${\rm det}(Y^TY)$ is a nonzero polynomial in the coordinates $x^{(\ell)}_k$, $k \in \mathcal{C}_i \cap \mathcal{C}_j$ with integer coefficients. Since the configuration of the points is generic, ${\rm det}(Y^TY) \neq 0$ yielding the linear independence of the columns of $Y$. Consequently, Eq.~\eqref{eq:represent} implies that $\beta^{(\ell)}_i = \beta^{(\ell)}_j$ for any two adjacent nodes in $G(n,r/2)$. Given that $r > 10 \sqrt{d} (\log n / n)^{1/d}$, the graph $G(n,r/2)$ is connected w.h.p. and thus the coefficients $\beta^{(\ell)}_i$ are the same for all $i$. Dropping subscript $(i)$, we obtain

\begin{equation*}
v = \sum_{\ell=1}^d \beta^{(\ell)} x^{(\ell)} + \beta^{(d+1)} u,
\end{equation*} 
proving that $\text{Ker}(\tilde{\Omega}) \subseteq \langle u, x^{(1)},\cdots, x^{(d)}\rangle$, and thus $\text{rank}(\tilde{\Omega}) \ge n -d -1$.  

\section{Proof of Claim~\ref{claim:Laplacian}}
\label{App:Laplacian}
Let $\tilde{G} = (V,\tilde{E})$, where $\tilde{E} = \{(i,j): d_{ij} \leq r/2\}$. The Laplacian of $\tilde{G}$ is denoted by $\tilde{\mathcal{L}}$. We first show that for some constant $C$, 
\begin{equation}
\label{eq:comp_Lt_P}
\tilde{\L} \preceq C \sum_{k=1}^n P_{u_{\C_k}}^{\perp}.
\end{equation}

\noindent Note that,
\begin{align*}
\sum_{k=1}^{n} P_{u_{\C_k}}^{\perp} &=
\sum_{k=1}^{n} (I - \frac{1}{|\C_k|} u_{\C_k} u_{\C_k}^T)
= \sum_{k=1}^{n} \frac{1}{|\C_k|} \Big(\sum_{i,j \in \C_k} M_{ij} \Big)\\
& \succeq \sum_{(i,j) \in \tilde{E}} \Big( \sum_{k:(i,j) \in \C_k} \frac{1}{|\C_k|}\Big) M_{ij}
= \sum_{(i,j) \in \tilde{E}} \Big( \sum_{k \in \C_i \cap \C_j} \frac{1}{|\C_k|}\Big) M_{ij}.
\end{align*}
The inequality follows from the fact that $M_{ij} \succeq \bf{0}$, $\forall i,j$. By application of Remark~\ref{rem:region}, we have $|\C_k| \leq C_1(nr^d)$ and $|\C_i \cap \C_j| \geq C_2nr^d$, for some constants $C_1$ and $C_2$ (depending on $d$) and $\forall k, i, j$. Therefore,
\begin{equation*}
\sum_{k=1}^{n} P_{u_{\C_k}}^{\perp} \succeq \sum_{(i,j) \in \tilde{E}} \frac{C_2}{C_1} M_{ij} = \frac{C_2}{C_1} \tilde{\mathcal{L}}.
\end{equation*}

Next we prove that for some constant $C$,
\begin{equation}
\label{eq:comp_L_Lt1}
\L \preceq C \tilde{\L}.
\end{equation}

\noindent To this end, we use the Markov chain comparison technique.

A path between two nodes $i$ and $j$, denoted by $\gamma_{ij}$, is a sequence of nodes $(i,v_1,\cdots,v_{t-1},j)$, such that the consecutive pairs are connected in $\tilde{G}$. Let $\gamma = (\gamma_{ij})_{(i,j) \in E}$ denote a collection of paths for all pairs connected in $G$, and let $\Gamma$ be the collection of all possible $\gamma$. Consider the probability distribution induced on $\Gamma$ by choosing paths between all connected pairs in $G$ in the following way.

Cover the space $[-0.5,0.5]^d$ with bins of side length $r/(4\sqrt{d})$ (similar to the proof of Proposition~\ref{pro:sampling_lemma}. As discussed there, w.h.p., every bin contains at least one node).  Paths are selected independently for different node pairs. Consider a particular pair $(i,j)$ connected in $G$. Select $\gamma_{ij}$ as follows. If $i$ and $j$ are in the same bin or in the neighboring bins then $\gamma_{ij} = (i,j)$. Otherwise, consider all bins intersecting the line joining $i$ and $j$.  From each of these bins, choose a node $v_k$ uniformly at random. Then the path $\gamma_{ij}$ is $(i,v_1,\cdots,j)$.

In the following, we compute the average number of paths passing through each edge in $\tilde{E}$. The total number of paths is $|E| = \Theta(n^2r^d)$. Also, since any connected pair in $G$ are within distance $r$ of each other and the side length of the bins is $O(r)$, there are $O(1)$ bins intersecting a straight line joining a pair $(i,j) \in E$. Consequently, each path contains $O(1)$ edges. The total number of bins is $\Theta(r^{-d})$. Hence, by symmetry, the number of paths passing through each bin is $\Theta(n^2 r^{2d})$. Consider a particular bin $B$ and the paths passing through it. All these paths are equally likely to choose any of the nodes in $B$. Therefore, the average number of paths containing a particular node in $B$, say $i$, is $\Theta(n^2 r^{2d} / nr^d) = \Theta(nr^d)$. In addition, the average number of edges between $i$ and neighboring bins is $\Theta(nr^d)$. Due to symmetry, the average number of paths containing an edge incident on $i$ is $\Theta(1)$. Since this is true for all nodes $i$, the average number of paths containing an edge is $\Theta(1)$. 

Now, let $v \in \reals^{n}$ be an arbitrary vector. For a directed edge $e \in \tilde{E}$ from $i \to j$, define $v(e) = v_i - v_j$. Also, let $|\gamma_{ij}|$ denote the length of the path $\gamma_{ij}$.
\begin{align}
v^T \L v &= \sum_{(i,j) \in E} (v_i - v_j)^2 = \sum_{(i,j) \in E} \bigg(\sum_{e \in \gamma_{ij}} v(e) \bigg)^2 \nonumber\\
& \leq \sum_{(i,j) \in E} |\gamma_{ij}| \sum_{e \in \gamma_{ij}} v(e)^2 =
\sum_{e \in \tilde{E}} v(e)^2 \sum_{\gamma_{ij} \ni e } |\gamma_{ij}| \nonumber\\
&\leq \gamma_* \sum_{e \in \tilde{E}} v(e)^2 b(\gamma,e), \label{eq:b(gamma,e)}
\end{align}
where $\gamma_*$ is the maximum path lengths and $b(\gamma,e)$ denotes the number of paths passing through $e$ under $\gamma = (\gamma_{ij})$. The first inequality follows from the Cauchy-Schwartz inequality. Since all paths have length $O(1)$, we have $\gamma_* =O(1)$. Also, note that in Eq.~\eqref{eq:b(gamma,e)}, $b(\gamma,e)$ is the only term that depends on the paths.  Therefore, we can replace $b(\gamma,e)$ with its expectation under the distribution on $\Gamma$, i.e., $b(e) = \sum_{\gamma \in \Gamma} \P(\gamma) b(\gamma,e)$. We proved above that the average number of paths passing through an edge is $\Theta(1)$. Hence, $\underset{e \in \tilde{E}}{\max}\; b(\gamma,e) = \Theta(1)$. using these bounds in Eq.~\eqref{eq:b(gamma,e)}, we obtain
\begin{equation}
\label{eq:comp_L_Lt2}
v^T \L v \leq C \sum_{e \in \tilde{E}} v(e)^2 = C v^T \tilde{\L}v,
\end{equation}  
for some constant $C$ and all vectors $v \in \reals^n$. Combining Eqs.~\eqref{eq:comp_Lt_P} and~\eqref{eq:comp_L_Lt2} implies the thesis.
 
\section{Proof of Claim~\ref{claim:concentration_points}}
\label{App:concentration_points}
%
In Remark~\ref{rem:region}, let region $\mathcal{R}$ be the $r/2$-neighborhood of node $i$, and take $c=2$. Then, with probability at least $1 - 2/n^2$, 
\begin{equation}
\label{eq:bound_Ci}
|\Ci| \in np_d + [- \sqrt{4np_d \log n},  \sqrt{4np_d \log n}],
\end{equation}
where $p_d = K_d(r/2)^d$.
\noindent Similarly, with probability at least $1 - 2/n^2$, 
\begin{equation}
\label{eq:bound_tCi}
|\tilde{\Ci}| \in n\tilde{p}_d + [- \sqrt{4n\tilde{p}_d \log n},  \sqrt{4n\tilde{p}_d \log n}],
\end{equation}
where $\tilde{p}_d = K_d(\frac{r}{2})^d (\frac{1}{2}+\frac{1}{100})^d$. By applying union bound over all $1\leq i \leq n$, Eqs.~\eqref{eq:bound_Ci} and~\eqref{eq:bound_tCi} hold for any $i$, with probability at least $1-4/n$. Given that $r > 10 \sqrt{d} (\log n / n)^{\frac{1}{d}}$, the result follows after some algebraic manipulations.

\section{Proof of Claim~\ref{claim:Gstar}}
\label{App:Gstar}
Part $(i)$: Let $\tilde{G} = (V,\tilde{E})$, where $\tilde{E} = \{(i,j): d_{ij} \leq r/2\}$. Also, let $A_{\tilde{G}}$ and $A_{G^*}$ respectively denote the adjacency matrices of the graphs $\tilde{G}$ and $G^*$. Therefore, $A_{\tilde{G}} \in \reals^{n \times n}$ and $A_{G^*} \in \reals^ {N \times N}$, where $N = |V(G^*)| = n(m+1)$. From the definition of $G^*$, we have
\begin{equation}
\label{eq:kron_adj}
A_{G^*} = A_{\tilde{G}} \otimes B, \quad
B = \begin{bmatrix}
1 & \cdots & 1\\
\vdots & \cdots & \vdots\\
1 & \cdots &1
\end{bmatrix}_{(m+1) \times (m+1)} 
\end{equation}
where $\otimes$ stands for the Kronecher product. Hence,
\begin{equation*}
\max_{i \in V(G^*)} \text{deg}_{G^*}(i)  = (m+1)  \max_{i \in V(\tilde{G})} \text{deg}_{\tilde{G}}(i).
\end{equation*}

\noindent Since the degree of nodes in $\tilde{G}$ are bounded by $C(nr^d)$ for some constant $C$, and $m \leq C(nr^d)$ (by definition of $m$ in Claim~\ref{claim:concentration_points}), we have that the degree of nodes in $G^*$ are bounded by $C(nr^d)^2$, for some constant $C$. 

\bigskip
\noindent Part $(ii)$: Let $D_{\tilde{G}} \in \reals^{n \times n}$ be the diagonal matrix with degrees of the nodes in $\tilde{G}$ on its diagonal. Define $D_{G^*} \in \reals^{N \times N}$ analogously. From Eq.~\eqref{eq:kron_adj}, it is easy to see that
\begin{equation*}
(D_{\tilde{G}}^{-1/2} A_{\tilde{G}} D_{\tilde{G}}^{-1/2}) \otimes (\frac{1}{m+1} B) = D^{-1/2}_{G^*} A_{G^*} D^{-1/2}_{G^*}.
\end{equation*}  
Now for any two matrices $\mathcal{A}$ and $\mathcal{B}$, the eigenvalues of $\mathcal{A} \otimes \mathcal{B}$ are all products of eigenvalues of $\mathcal{A}$ and $\mathcal{B}$. The matrix $1/(m+1) B$ has eigenvalues $0$, with multiplicity $m$, and $1$, with multiplicity one. Thereby,
\begin{equation*}
\sigma_{\min}(I - D^{-1/2}_{G^*} A_{G^*}D^{-1/2}_{G^*}) \geq \min\{\sigma_{\min}(I - D_{\tilde{G}}^{-1/2} A_{\tilde{G}} D_{\tilde{G}}^{-1/2}), 1\} \geq Cr^2,
\end{equation*} 
where the last step follows from Remark~\ref{rem:spec_GNR}.
Due to the result of~\cite{butler2008eas} (Theorem 4), we obtain
\begin{equation*}
\sigma_{\min} (\L_{G^*}) \geq d_{\min,G^*}\sigma_{\min}(\L_{n,G^*}),
\end{equation*}
where $d_{\min,G^*}$ denotes the minimum degree of the nodes in $G^*$, and $\L_{n,G^*} = I - D^{-1/2}_{G^*} A_{G^*}D^{-1/2}_{G^*}$ is the normalized Laplacian of $G^*$.  Since $d_{\min,G^*} = (m+1) d_{\min,\tilde{G}} \geq C(nr^d)^2$, for some constant $C$, we obtain
\begin{equation*}
\sigma_{\min}(\L_{G^*}) \geq C(nr^d)^2r^2,
\end{equation*}
for some constant $C$.

\section{Proof of Claim~\ref{claim:Azuma_bound}}
\label{App:Azuma_bound}

Fix a pair $(i,j) \in E(G^*)$. Let $m_{ij} = |\Q_i \cap \Q_j|$, and without loss of generality assume that the nodes in $\Q_i \cap \Q_j$ are labeled with $\{1,\cdots, m_{ij}\}$. Let $z^{(\ell)} = \txl_{\Q_i \cap \Q_j}$, for $1 \le \ell \le d$, and let $z_k = (z_k^{(1)},\cdots, z_k^{(d)})$, for $1 \le k \le m_{ij}$. Define the matrix $M^{(ij)} \in \reals^{d \times d}$ as $M^{(ij)}_{\ell,\ell'} = \langle z^{(\ell)}, z^{(\ell')} \rangle$, for $1\le \ell', \ell \le d$. Finally, let $\beta_{ij} = (\beta^{(1)}_j - \beta^{(1)}_i,\cdots,\beta^{(d)}_j - \beta^{(d)}_i) \in \reals^d$. Then,
\begin{equation}
\label{eqn:Mij_betaij}
\|\sum_{\ell=1}^d (\beta^{(\ell)}_j - \beta^{(\ell)}_i) \txl_{\Q_i \cap \Q_j}\|^2
= \beta_{ij}^T M^{(ij)} \beta_{ij} \ge \sigma_{\min}(M^{(ij)}) \|\beta_{ij}\|^2.
\end{equation}

In the following, we lower bound $\sigma_{\min}(M^{(i,j)})$. Notice that
\begin{equation}
\label{eqn:Mij_def}
M^{(ij)} = \sum_{k=1}^{m_{ij}} z_k z_k^T = \sum_{k=1}^{m_{ij}} \{z_kz_k^T - \E(z_kz_k^T)\} + \sum_{k=1}^{m_{ij}} \E(z_kz_k^T).
\end{equation}

We first lower bound the quantity $\sigma_{\min}(\sum_{k=1}^{m_{ij}} \E(z_k z_k^T))$. Let $S \in \reals^{d \times d}$ be an orthogonal matrix that aligns the line segment between $x_i$ and $x_j$ with $e_1$. Now, let $\hat{z_k} = S z_k$ for $1 \le k \le m_{ij}$. Then,
\begin{equation*}
\sum_{k=1}^{m_{ij}} \E(z_k z_k^T) = \sum_{k=1}^{m_{ij}} S^T \E(\hat{z}_k \hat{z}_k^T) S.
\end{equation*}
The matrix $\E(\hat{z}_k \hat{z}_k^T)$ is the same for all $1 \le k \le m_{ij} $. Further, it is a diagonal matrix whose diagonal entries are bounded from below by $C_1r^2$, for some constant $C_1$. Therefore, $\sigma_{\min}(\sum_{k=1}^{m_{ij}} \E(\hat{z}_k \hat{z}_k^T)) \ge m_{ij} C_1r^2$. Consequently, 
\begin{equation}
\label{eqn:term1}
\sigma_{\min} (\sum_{k=1}^{m_{ij}} \E(z_kz_k^T)) = \sigma_{\min}(\sum_{k=1}^{m_{ij}} \E(\hat{z}_k \hat{z}_k^T)) \ge m_{ij} C_1r^2.
\end{equation}

Let $Z^{(k)} = z_kz_k^T - \E(z_kz_k^T)$, for $1 \le k \le m_{ij}$. Next, we upper bound the quantity $\sigma_{\max}(\sum_{k=1}^{m_{ij}} Z^{(k)})$. Note that for any matrix $A \in \reals ^{d \times d}$,
\begin{align*}
\sigma_{\max}(A) &= \max_{\|x\|=\|y\|=1} x^T A y   \le \max_{\|x\|=\|y\|=1} \sum_{1\le p,q \le d} |A_{pq}| |x_py_q|\\
& \le \max_{1 \le p,q \le d} |A_{pq}| \cdot \max_{\|x\|=1}(\sum_{p=1}^d |x_p|) \cdot \max_{\|y\|=1}(\sum_{q=1}^d |y_q|) \le d \max_{1 \le p,q \le d} |A_{pq}|. 
\end{align*}
Taking $A = \sum_{k=1}^{m_{ij}} Z^{(k)}$, we have
\begin{equation}
\label{eqn:l_2l_inf}
\P\Big(\sigma_{\max}(\sum_{k=1}^{m_{ij}} Z^{(k)}) > \epsilon \Big)
 \le \P \Big( \max_{1\le p, q\le d} \Big|\sum_{k=1}^{m_{ij}} Z^{(k)}_{pq} \Big| > \frac{\epsilon}{d} \Big)
 \le d^2 \max_{1 \le p,q \le d} \P \Big( \Big| \sum_{k=1}^{m_{ij}} Z^{(k)}_{pq} \Big| > \frac{\epsilon}{d}\Big),
\end{equation}
where the last inequality follows from union bound. Take $\epsilon = C_1m_{ij}r^2/2$. Note that $\{Z^{(k)}_{pq}\}_{1\le k \le m_{ij}}$ is a sequence of independent random variables with $\E(Z^{(k)}_{pq}) = 0$, and $|Z^{(k)}_{pq}| \le r^2/4$, for $1\le k \le m_{ij}$. Applying Hoeffding 's inequality,
\begin{align}
\label{eqn:hoeff_bound}
\P \Big( \Big| \sum_{k=1}^{m_{ij}} Z^{(k)}_{pq} \Big| > \frac{C_1 m_{ij}r^2}{2d} \Big) \le
2 \exp \Big( -\frac{2C_1^2 m_{ij}}{d^2}\Big) \leq 2 n^{-3}.
\end{align}
Combining Eqs.~\eqref{eqn:l_2l_inf} and~\eqref{eqn:hoeff_bound}, we obtain
\begin{align}
\label{eqn:term2}
\P \Big( \sigma_{\max}(\sum_{k=1}^{m_{ij}} Z^{(k)}) >  \frac{C_1m_{ij}r^2}{2}\Big) \le 2 d^2 n^{-3}.
\end{align}

\noindent Using Eqs.~\eqref{eqn:Mij_def},~\eqref{eqn:term1} and~\eqref{eqn:term2}, we have
\begin{align*}
\sigma_{\min}(M^{(ij)}) \ge \sigma_{\min}(\sum_{k=1}^{m_{ij}} \E(z_kz_k^T))
 - \sigma_{\max} (\sum_{k=1}^{m_{ij}} Z^{(k)}) \ge \frac{C_1 m_{ij} r^2}{2},
\end{align*}
with probability at least $1- 2d^2n^{-3}$. Applying union bound over all pairs $(i,j) \in E(G^*)$, we obtain that w.h.p., $\sigma_{\min}(M^{(ij)}) \ge C_1 m_{ij} r^2/2 \ge C(nr^d)r^2$, for all $(i,j) \in E(G^*)$. Invoking Eq.~\eqref{eqn:Mij_betaij},
\begin{equation*}
\|\sum_{\ell=1}^d (\beta^{(\ell)}_j - \beta^{(\ell)}_i) \txl_{\Q_i \cap \Q_j}\|^2
\ge C(nr^d) r^2 \|\beta_{ij}\|^2 = C(nr^d)r^2 \sum_{\ell=1}^d (\beta_j^{(\ell)} - \beta_i^{(\ell)})^2.
\end{equation*}
\section{Proof of Claim~\ref{claim:beta_mean}}
\label{App:beta_mean}
\begin{proof}
Let $N = |V(G^*)| = n(m+1)$. Define $\bar{\beta}^{(\ell)} =  (1/N) \sum_{i=1}^N \beta_i^{(\ell)}$ and let $\tilde{v} = v - \sum_{\ell=1}^d \bar{\beta}^{(\ell)} x^{(\ell)}$. Then, the vector $\tilde{v}$ has the following local decompositions.
\begin{equation*}
\tilde{v}_{\Q_i} = \sum_{\ell=1}^{d} (\bl_i - \bar{\beta}^{(\ell)}) \txl_{\Q_i} + \tilde{\gamma}_i u_{\Q_i} + w^{(i)},
\end{equation*}

\noindent where $\tilde{\gamma}_i = \gamma_i - \sum_{\ell=1}^d \bar{\beta}^{(\ell)} \frac{1}{|\Q_i|} \langle x^{(\ell)}_{\Q_i}, u_{\Q_i} \rangle$. For convenience, we establish the following definitions.
\\
$M \in \reals^{d \times d}$ is a matrix with $M_{\ell,\ell'} = \langle \xl, x^{(\ell')} \rangle$. Also, for any $1\leq i \leq N$, define the matrix $M^{(i)} \in \reals^{d \times d}$ as $M^{(i)}_{\ell, \ell'} = \langle \txl_{\Q_i},\tilde{x}^{(\ell')}_{\Q_i} \rangle$. Let $\hat{\beta}^{(\ell)}_i := \bl_i - \bar{\beta}^{(\ell)}$ and $\eta^{(\ell)}_i = \sum_{\ell'} M^{(i)}_{\ell,\ell'} \hb^{(\ell')}_i $. Finally, for any $1 \leq \ell \leq d$, define the matrix $B^{(\ell)} \in \reals^{N \times n}$ as follows.
\begin{equation*}
B^{(\ell)}_{i,j} = 
\begin{cases}
\txl_{\Q_i,j} & \text{if} \; j \in \Q_i\\
0 & \text{if}\; j \notin \Q_i
\end{cases}
\end{equation*}
Now, note that $\langle \tilde{v}_{\Q_i}, \txl_{\Q_i}\rangle = \sum_{\ell'=1}^{d} M^{(i)}_{\ell,\ell'} \hb^{(\ell')}_i = \eta^{(\ell)}_i $. Writing it in matrix form, we have $B^{(\ell)} \tilde{v} = \eta^{(\ell)}$.

Our first lemma provides a lower bound for $\sigma_{\min}(B^{(\ell)})$. For its proof, we refer to Section~\ref{lem:comp_B_L}.

\begin{lem}
\label{lem:comp_B_L}
Let $\tilde{G} = (V,\tilde{E})$, where $\tilde{E} = \{(i,j):d_{ij} \leq r/2\}$ and denote by $\tilde{\L}$ the Laplacian of $\tilde{G}$. Then, there exists a constant $C = C(d)$, such that, w.h.p.
\begin{equation*}
B^{(\ell)} (B^{(\ell)})^T \succeq C(nr^d)^{-1} r^2\tilde{\L}, \quad \forall 1\leq \ell \leq d.
\end{equation*}
\end{lem}

Next lemma establishes some properties of the spectral of the matrices $M$ and $M^{(i)}$. Its proof is deferred to Section~\ref{sec:concentration_M}.

\begin{lem}
\label{lem:concentration_M}
There exist constants $C_1$ and $C_2$, such that, w.h.p.
\begin{equation*}
\sigma_{\min}(M) \geq C_1 n, \quad \quad \sigma_{\max}(M^{(i)}) \leq C_2(nr^d)r^2,
\quad \forall\; 1\leq i \leq N.
\end{equation*}
\end{lem}

\noindent Now, we are in position to prove Claim~\ref{claim:beta_mean}.
Using Lemma~\ref{lem:comp_B_L} and since $\langle \tilde{v}, u\rangle = 0$,
\begin{equation*}
\|\eta^{(\ell)}\|^2 \geq \sigma_{\min}(B^{(\ell)}(B^{(\ell)})^T) \|\tilde{v}\|^2
\ge C(nr^d)^{-1} r^2 \sigma_{\min}(\tilde{\L}) 
\geq C r^4 \|\tilde{v}\|^2,
\end{equation*}
for some constant $C$. The last inequality follows from the lower bound on $\sigma_{\min}(\tilde{\L})$ provided by Remark~\ref{rem:spec_GNR}. Moreover,
\begin{align*}
\bigg[ \sum_{\ell'=1}^d M_{\ell,\ell'} \bar{\beta}^{(\ell')} \bigg]^2 = 
\langle \tilde{v}, \xl \rangle^2 \leq \|\tilde{v}\|^2 \|\xl\|^2
\leq C r^{-4} \|\eta^{(\ell)}\|^2 \|\xl\|^2.
\end{align*}
Summing both hand sides over $\ell$ and using $\|\xl\|^2 \leq Cn$, we obtain
\begin{equation*}
\sum_{\ell=1}^{d} \bigg[ \sum_{\ell'=1}^d M_{\ell,\ell'} \bar{\beta}^{(\ell')} \bigg]^2
\leq C (nr^{-4}) \sum_{\ell=1}^{d} \|\eta^{(\ell)}\|^2.
\end{equation*}
Equivalently,
\begin{equation*}
\sum_{\ell=1}^{d} \langle M_{\ell,\cdot}, \bar{\beta}\rangle ^2
\leq C (nr^{-4}) \sum_{\ell=1}^{d}
\sum_{i=1}^{N} \langle M^{(i)}_{\ell,\cdot}, \hb_i\rangle^2.
\end{equation*}
Here, $\bar{\beta} = (\bar{\beta}^{(1)},\cdots, \bar{\beta}^{(d)}) \in \reals ^d$ and $\hb_i = (\hb_i^{(1)},\cdots,\hb_i^{(d)}) \in \reals^d$. Writing this in matrix form,
\begin{equation*}
\|M \bar{\beta}\|^2 \leq C  (nr^{-4}) \sum_{i=1}^{N}
\|M^{(i)} \hb_i\|^2.
\end{equation*}
Therefore,
\begin{equation*}
\sigma_{\min}^2(M) \|\bar{\beta}\|^2 \leq C (nr^{-4}) \bigg[\max_{1 \leq i \leq N}
\sigma_{\max}^2(M^{(i)}) \bigg] \sum_{i=1}^N \|\hb_i\|^2.
\end{equation*}
Using the bounds on $\sigma_{\min}(M)$ and $\sigma_{\max}(M^{(i)})$ provided in Lemma~\ref{lem:concentration_M}, we obtain 
\begin{equation}
\label{eq:bound_beta_check}
\|\bar{\beta}\|^2 \leq \frac{C}{n} (nr^d)^2 \sum_{i=1}^N \|\hb_i\|^2.
\end{equation}
Now, note that
\begin{align}
&\|\bar{\beta}\|^2  = \sum_{\ell=1}^d (\bar{\beta}^{(\ell)})^2 = 
\sum_{\ell=1}^d \bigg( \frac{\sum_{i=1}^N \bl_i}{N}\bigg)^2=
\frac{1}{N} \sum_{\ell=1}^d \|P_{u} \beta^{(\ell)}\|^2,\label{eq:beta_check_eq}\\
&\sum_{i=1}^N \|\hb_i\|^2 = \sum_{\ell=1}^d \sum_{i=1}^N (\bl_i - \bar{\beta}^{(\ell)})^2 = 
\sum_{\ell=1}^{d} \|P_{u}^{\perp} \beta^{(\ell)}\|^2. \label{eq:beta_hat}
\end{align}
Consequently,
\begin{align*}
\sum_{\ell=1}^{d} \|\bl\|^2 &= \sum_{\ell=1}^d \|P_u \beta^{(\ell)}\|^2 + \sum_{\ell=1}^d \|P^{\perp}_u \beta^{(\ell)}\|^2\\
& \stackrel{(a)}{=} N \|\bar{\beta}\|^2 + \sum_{\ell=1}^d \|P^{\perp}_u \beta^{(\ell)}\|^2\\
& \stackrel{(b)}{\le} \frac{CN}{n} (nr^d)^2 \sum_{i=1}^N \|\hb_i\|^2 +  \sum_{\ell=1}^d \|P^{\perp}_u \beta^{(\ell)}\|^2\\
& \stackrel{(c)}{=} (1 + \frac{CN}{n}(nr^d)^2) \sum_{\ell=1}^{d} \|P_{u}^{\perp} \beta^{(\ell)}\|^2 \\
& \leq C(nr^d)^3 \sum_{\ell=1}^{d} \|P_{u}^{\perp} (\beta^{(\ell)})\|^2.
\end{align*}
Here, $(a)$ follows from Eq.~\eqref{eq:beta_check_eq}; $(b)$ follows from Eq.~\eqref{eq:bound_beta_check} and $(c)$ follows from Eq.~\eqref{eq:beta_hat}. The result follows.
\end{proof}

\subsection{Proof of Lemma~\ref{lem:comp_B_L}}
Recall that $e_{ij} \in \reals^{n}$ is the vector with $+1$ at the $i^{th}$ position, $-1$ at the $j^{th}$ position and zero everywhere else. For any two nodes $i$ and $j$ with $\|x_i - x_j\| \leq r/2$, choose a node $k \in \tilde{\C}_i \cap \tilde{\C}_j$ uniformly at random and consider the cliques $\Q_1 = \C_k$, $\Q_2 = \C_k \backslash{i}$, and $\Q_3 = \C_k \backslash{j}$. Define $S_{ij} = \{\Q_1,\Q_2,\Q_3\}$. Note that $S_{ij} \subset \cliq^*(G)$.

Let $a_1$, $a_2$ and $a_3$ respectively denote the center of mass of the points in cliques $\Q_1$, $\Q_2$ and $\Q_3$. Find scalars $\xi^{(ij)}_1$, $\xi^{(ij)}_2$, and $\xi^{(ij)}_3$, such that
\begin{equation}
\label{eq:weights}
\begin{split}
\begin{cases}
\xi^{(ij)}_1 + \xi^{(ij)}_2 + \xi^{(ij)}_3 = 0,\\
\xi^{(ij)}_1 a^{(\ell)}_1 + \xi^{(ij)}_2 a^{(\ell)}_2 + \xi^{(ij)}_3 a^{(\ell)}_3 = 0,\\
\xi^{(ij)}_1 (x^{(\ell)}_i -  a^{(\ell)}_1) + \xi^{(ij)}_3 (x^{(\ell)}_i - a^{(\ell)}_3) = 1.
\end{cases}
\end{split}
\end{equation}
Note that the space of the solutions of this linear system of equations is invariant to translation of the points. Hence, without loss of generality, assume that $\underset{l \in \Q_1, l \neq i,j}{\sum} x_l = 0$. Also, let $m = |\C_k|$. Then, it is easy to see that
\begin{align*}
 a_1 = \frac{x_i+x_j}{m}, \quad a_2 = \frac{x_j}{m}, \quad a_3 = \frac{x_i}{m},
 \end{align*}
 and the solution of Eqs.~\eqref{eq:weights} is given by
 \begin{align*}
 \xi^{(ij)}_1 = \frac{\xl_j - \xl_i} {\xl_j ( \xl_i - \dfrac{\xl_j}{m} )}, \quad
\xi^{(ij)}_2 = -\frac{\xl_j} {\xl_j( \xl_i -\dfrac{\xl_j}{m} )}, \quad
\xi^{(ij)}_3 = \frac{\xl_i}{\xl_j ( \xl_i - \dfrac{\xl_j}{m} )}.
\end{align*} 

Firstly, observe that\\

\noindent $\bullet \quad \xi^{(ij)}_1 (\xl_i -  a^{(\ell)}_1) + \xi^{(ij)}_2 (\xl_i - a^{(\ell)}_3) = 1$.\\
\noindent $\bullet \quad \xi^{(ij)}_1 (\xl_j -  a^{(\ell)}_1) + \xi^{(ij)}_2 (\xl_j - a^{(\ell)}_2) = -1$.\\
\noindent $\bullet \quad$ For $t \in \C_k$ and $t \neq i,j$: 
\begin{align*}
&\xi^{(ij)}_1 (\xl_t -  a^{(\ell)}_1) + \xi^{(ij)}_2 (\xl_t -  a^{(\ell)}_2) + \xi^{(ij)}_3 (\xl_t -  a^{(\ell)}_3)\\ 
&= (\xi^{(ij)}_1 +  \xi^{(ij)}_2 + \xi^{(ij)}_3) \xl_t - 
(\xi^{(ij)}_1a^{(\ell)}_1+\xi^{(ij)}_2a^{(\ell)}_2+\xi^{(ij)}_3a^{(\ell)}_3) = 0.
\end{align*}

Therefore,
\begin{equation}
\label{eqn:ksi_mat}
\xi^{(ij)}_1 \txl_{\Q_1,t} +\xi^{(ij)}_2 \txl_{\Q_2,t} +\xi^{(ij)}_3 \txl_{\Q_3,t} =
\begin{cases}
1 & \text{if } t =i,\\
-1 & \text{if } t = j,\\
0 & \text{if } t \in \C_k, t\neq i,j
\end{cases}
\end{equation}

Let $\xi^{(ij)} \in \reals^{N}$ be the vector with $\xi^{(ij)}_1$, $\xi^{(ij)}_2$ and $\xi^{(ij)}_3$ at the positions corresponding to the cliques $\Q_1$, $\Q_2$, $\Q_3$ and zero everywhere else. Then, Eq.~\eqref{eqn:ksi_mat} gives $(B^{(\ell)})^T \xi^{(ij)} = e_{ij}$.

Secondly, note that $\|\xi^{(ij)}\|^2 = (\xi^{(ij)}_1)^2 + (\xi^{(ij)}_2)^2 +(\xi^{(ij)}_3)^2 \leq \dfrac{C}{r^2}$, for some constant $C$.

Now, we are in position to prove Lemma~\ref{lem:comp_B_L}.
\\
For any vector $z \in \reals^n$, we have
\begin{align*}
\langle z, \tilde{\L} z \rangle &= 
\sum_{(i,j) \in \tilde{E}} \langle e_{ij}, z \rangle^2 =
\sum_{(i,j) \in \tilde{E}} \langle \xi^{(ij)}, B^{(\ell)}z \rangle^2
= \sum_{(i,j) \in \tilde{E}} \bigg( \sum_{\Q_t \in S_{ij}} \xi^{(ij)}_t \langle B^{(\ell)}_{\Q_t,\cdot}, z \rangle\bigg)^2\\
& \leq \sum_{(i,j) \in \tilde{E}} \bigg( \sum_{\Q_t \in S_{ij}} [\xi^{(ij)}_t]^2\bigg)
\bigg( \sum_{\Q_t \in S_{ij}} \langle B^{(\ell)}_{\Q_t,\cdot}, z \rangle ^ 2\bigg) 
\leq \max_{(i,j) \in \tilde{E}} \|\xi^{(ij)}\|^2 \sum_{\Q_t} \langle B^{(\ell)}_{\Q_t,.},
z \rangle^2 ( \sum_{S_{ij} \ni \Q_t} 1)\\
&\leq \frac{C}{r^2} (nr^d) \|B^{(\ell)} z\|^2.
\end{align*}
Hence, $B^{(\ell)} (B^{(\ell)})^T \succeq C(nr^d)^{-1} r^2 \tilde{\L}$.

\subsection{Proof of Lemma~\ref{lem:concentration_M}}
\label{sec:concentration_M}
First, we prove that $\sigma_{\min}(M) \geq C n$, for some constant $C$.

 By definition, $M = \sum_{i=1}^n x_i x_i^T$. Let $Z_i = x_i x_i^T \in \reals^{d \times d}$, and $\bar{Z} = 1/n \sum_{i=1}^n Z_i$. Note that $\{Z_i\}_{1 \le i \le n}$ is a sequence of i.i.d. random matrices with $Z = \E(Z_i) = 1/12 I_{d \times d}$. By Law of large number we have $\bar{Z} \to Z$, almost surely. In addition, since $\sigma_{\max}(.)$ is a continuos function of its argument, we obtain $\sigma_{\max}(\bar{Z} - Z) \to 0$, almost surely. Therefore,

\begin{align*}
\sigma_{\min}(M) = n\sigma_{\min}(\bar{Z}) \ge n\Big(\sigma_{\min}(Z) - \sigma_{\max}(\bar{Z} - Z)\Big) = n\Big(\frac{1}{12} - \sigma_{\max}(\bar{Z} - Z)\Big),
\end{align*}
whence we obtain $\sigma_{\min}(M) \ge n/12$, with high probability.

%
%
Now we pass to proving the second part of the claim.
\\
Let $m_i = |\Q_i|$, for $1 \leq i \leq N$. Since $M^{(i)} \succeq 0$, we have\begin{equation*}
\sigma_{\max}(M^{(i)}) \le \Tr(M^{(i)}) = \sum_{\ell=1} \|\txl_{\Q_i}\|^2  \leq Cm_ir^2.
\end{equation*} 
With high probability, $ m_i \le C(nr^d)$, $\forall 1 \leq i \leq N$, and for some constant $C$. Hence, 
\begin{equation*}
\max_{1 \leq i \leq N} \sigma_{\max}(M^{(i)}) \leq C(nr^d)r^2,
\end{equation*}
with high probability. The result follows.
\section{Proof of Proposition~\ref{pro:aux2}}
\label{App:aux2}
 \begin{proof}
  Recall that $\tR = XY^T + Y X^T$ with $X, Y \in \reals^{n \times d}$ and $Y^T u = 0$. By triangle inequality, we have 
  \begin{align*}
  |\langle x_i - x_j,y_i - y_j\rangle| &\geq
  |\langle x_i,y_j \rangle + \langle x_j,y_i \rangle | -
  |\langle x_i,y_i\rangle| - |\langle x_j,y_j\rangle|\\
  &= |\tR_{ij}| - \frac{|\tR_{ii}|}{2} - \frac{|\tR_{jj}|}{2}.
  \end{align*}
  Therefore,
  \begin{equation}
  \label{eq:triangle_bound1}
  \sum_{i,j} |\langle x_i - x_j,y_i -y_j\rangle| \geq \sum_{i,j} |\tR_{ij}| - n \sum_{i} |\tR_{ii}|.
  \end{equation}
  
  Again, by triangle inequality,
  \begin{align}
  \sum_{ij} |\langle x_i -x_j, y_i -y_j\rangle| &\geq
  \sum_{i} | n\langle x_i,y_i\rangle  + \sum_{j} \langle x_j,y_j\rangle
  -\langle x_i,\sum_{j} y_j \rangle    -  \langle \sum_{j} x_j,y_i \rangle|\nonumber\\
  &= n \sum_{i} |\langle x_i,y_i \rangle + \frac{1}{n} \sum_{j} \langle x_j,y_j\rangle|,
    \label{eq:triangle_bound2}
  \end{align}
  where the last equality follows from $Y^T u = 0$ and $X^T u = 0$.
 \begin{remark}
 \label{rem:bound3}
 For any $n$ real values $\xi_1,\cdots,\xi_n$, we have
 \begin{equation*}
 \sum_{i} |\xi_i + \bar{\xi}| \geq \frac{1}{2} \sum_{i} |\xi_i|,
 \end{equation*}
 where $\bar{\xi} = (1/n) \sum_{i} \xi_i$.
 \end{remark}
 
 \begin{proof}[Proof (Remark~\ref{rem:bound3})] 
Without loss of generality, we assume $\bar{\xi} \geq 0$. Then,
\begin{equation*}
\sum_{i} |\xi_i + \bar{\xi}| \geq \sum_{i: \xi_i \geq 0} \xi_i \geq
\frac{1}{2} (\sum_{i: \xi_i \geq0} \xi_i - \sum_{i: \xi_i < 0} \xi_i) = 
\frac{1}{2} \sum_{i} |\xi_i|,
\end{equation*}
where the second inequality follows from $\sum_{i} \xi_i = n \bar{\xi} \geq 0$.
\end{proof}
 
 Using Remark~\ref{rem:bound3} with $\xi_i = \langle x_i, y_i\rangle$, Eq.~\eqref{eq:triangle_bound2} yields
 \begin{equation}
 \label{eq:triangle_bound3}
  \sum_{ij} |\langle x_i -x_j, y_i -y_j\rangle| \geq
  \frac{n}{2} \sum_{i} |\langle x_i,y_i\rangle| = \frac{n}{4} \sum_{i} |\tR_{ii}|.
 \end{equation}
 Combining Eqs.~\eqref{eq:triangle_bound1} and~\eqref{eq:triangle_bound3}, we obtain
 \begin{equation}
 \|\cR_{K_n,X}(Y)\|_1 = \sum_{ij} |\langle x_i -x_j, y_i -y_j\rangle| \geq \frac{1}{5} \|\tR\|_1.
 \end{equation}
 which proves the desired result.
 \end{proof}
 
\section{Proof of Lemma~\ref{lem:congestion}}
\label{App:chain_exist}

We will compute the average number of chains passing through a particular edge in the order notation. Notice that the total number of chains is $\Theta(n^2)$ since there are $n \choose 2$ node pairs. Each chain has $O(1/r)$ vertices and thus intersects $O(1/r)$ bins. The total number of bins is $\Theta(1/r^d)$. Hence, by symmetry, the number of chains intersecting each bin is $\Theta (n^2 r^{d-1})$. Consider a particular bin $B$, and the chains intersecting it. Such chains are equally likely to select any of the nodes in $B$. Since the expected number of nodes in $B$ is $\Theta(nr^d)$, the average number of chains containing a particular node, say $i$, in $B$, is $\Theta(n^2 r^{d-1} / nr^{d}) = \Theta(nr^{-1})$. Now consider node $i$ and one of its neighbors in the chain, say $j$. Denote by $B^*$ the bin containing node $j$. The number of edges between $i$ and $B^*$ is $\Theta(nr^d)$.  Hence, by symmetry, the average number of chains containing an edge incident on $i$ will be $\Theta(n r^{-1}/ nr^d) = \Theta(r^{-d-1})$.  This is true for all nodes. Therefore, the average number of chains containing any particular edge is $O(r^{-d-1})$. In other words, on average, no edge belongs to more than $O(r^{-d-1})$ chains. 

\section{The Two-Part Procedure for General $d$}
\label{App: stage_1_G}
In proof of Lemma~\ref{lem:cheeger_aux}, we stated a two-part procedure to find the values $\{\lambda_{lk}\}_{(l,k) \in E(G_{ij})}$ that satisfy Eq.~\eqref{eq:constraints}. Part $(i)$ of the procedure was demonstrated for the special case $d=2$. Here, we discuss this part for general $d$. 

Let $G_{ij} = \{i\} \cup \{j\} \cup H_1 \cup \cdots \cup H_k$ be the chain between nodes $i$ and $j$. Let $\mathcal{F}_p = H_p \cap H_{p+1}$. Without loss of generality, assume $V(\mathcal{F}_p) = \{1,2,\cdots,q\}$, where $q = 2^{d-1}$.  The goal is to find a set of forces, namely $f_1,\cdots,f_q$, such that
\begin{equation}
\begin{split} \label{eq:force_constraint_G}
\sum_{i=1}^{q} f_i &= x_m, \quad 
\sum_{i=1}^{q} f_i \wedge x_i = 0,\\
&\sum_{i=1}^{q} \|f_i\|^2 \leq C \|x_m\|^2.
\end{split}
\end{equation}
%
It is more convenient to write this problem in matrix form. Let $X = [x_1, x_2, \cdots, x_{q}] \in \reals^{d \times q}$ and $\Phi = [f_1, f_2, \cdots, f_q] \in \reals^{d \times q}$. Then, the problem can be recast as finding a matrix $\Phi \in \reals ^{d \times d}$, such that,

\begin{equation}
\begin{split} \label{eq:opt_G}
\Phi u &= x_m, \quad 
X \Phi^T = \Phi X^T,\\
&\|\Phi\|_F^2 \leq C \|x_m\|^2.
\end{split}
\end{equation}

Define $\tilde{X} = X (I - 1/q u u^T)$, where $I \in \reals^{q \times q}$ is the identity matrix and $u \in \reals^q$ is the all-ones vector. Let
%
%
%
%
\begin{equation}
\label{eqn:phi_nominate}
\Phi = \frac{1}{q} x_m u^T + (\frac{1}{q} Xux_m^T + S) (\tilde{X}\tilde{X}^T)^{-1} \tilde{X},
\end{equation}
where $S \in \reals^{d \times d}$ is an arbitrary symmetric matrix. Observe that
\begin{equation}
\Phi u = x_m, \quad X \Phi^T = \Phi X^T.
\end{equation}
Now, we only need to find a symmetric matrix $S \in \reals^{d \times d}$ such that the matrix $\Phi$ given by 
Eq.~\eqref{eqn:phi_nominate} satisfies $\|\Phi\|_F \le C \|x_m\|$. Without loss of generality, assume that the vector $x_m$ is in the direction $e_1= (1,0,\cdots, 0) \in \reals^d$. Let $x_c = \frac{1}{q} X u$ be the center of the nodes $\{x_i\}_{i=1}^{q}$, and let $x_c = (x_c^{(1)},\cdots,x_c^{(d)})$. Take $S = - \|x_m\| x_c^{(1)} e_1 e_1^T$. From the construction of the chain $G_{ij}$, the nodes $\{x_i\}_{i=1}^{q}$ are obtained by wiggling the vertices of a hypercube aligned in the direction $x_m/\|x_m\| = e_1$, and with side length $\tilde{r} = 3r/4\sqrt{2}$. (each node wiggles by at most $\frac{r}{8}$). Therefore, $x_c$ is almost aligned with $e_1$, and has small components in the other directions. Formally, $|x_c^{(i)}| \leq \frac{r}{8}$, for $2 \leq i \leq d$. Therefore
\begin{align*}
\frac{1}{q} X u x_m^T + S &= (\sum_{i=1}^{d} x_c^{(i)} e_i) \cdot
(\|x_m\| e_1)^T - \|x_m\| x_c^{(1)} e_1 e_1^T\\ 
&= \sum_{i=2}^{d} \|x_m\| x_c^{(i)} e_i e_1 ^T.
\end{align*}
Hence, $\frac{1}{q} Xu x_m^T  + S \in \reals^{d \times d}$ has entries bounded by $\frac{r}{8} \|x_m\|$. In the following we show that there exists a constant $C= C(d)$, such that all entries of $(\tilde{X}\tilde{X}^T)^{-1} \tilde{X}$ are bounded by $C/r$. Once we show this, it follows that
\begin{equation*}
\|(\frac{1}{q} X u x_m^T +  S) (\tilde{X}\tilde{X}^T)^{-1} \tilde{X}\|_F \leq C \|x_m\|,
\end{equation*}
for some constant $C = C(d)$. Therefore,
\begin{equation*}
\|\Phi\|_F \leq \|\frac{1}{q} x_m u^T\|_F + \| (\frac{1}{q} X u x_m^T +  S) (\tilde{X}\tilde{X}^T)^{-1} \tilde{X}\|_F \leq C\|x_m\|, 
\end{equation*}
for some constant $C$.

We are now left with the task of showing that all entries of $(\tilde{X}\tilde{X}^T)^{-1} \tilde{X}$ are bounded by $C/r$, for some constant $C$.

The nodes $x_i$ were obtained by wiggling the vertices of a hypercube of side length $\tilde{r} = 3r / 4 \sqrt{2}$. (each node wiggles by at most $r/8$). Let $\{z_i\}_{i=1}^{q}$ denote the vertices of this hypercube, and thus $\|x_i - z_i\| \leq \frac{r}{8}$. Define
 \begin{equation*}
 Z = \frac{1}{\tilde{r}} [z_1,\cdots, z_q], \quad \quad \delta Z = \frac{1}{\tilde{r}}\tilde{X} - Z.
 \end{equation*}
Then, $\tilde{X} \tilde{X}^T = \tilde{r}^2 (Z+ \delta Z) (Z +\delta Z )^T  = \tilde{r}^2 (ZZ^T+ \bar{Z})$, where $\bar{Z} = Z (\delta Z)^T + (\delta Z) Z^T + (\delta Z)(\delta Z)^T$. 
Consequently,
\begin{equation*}
(\tilde{X} \tilde{X}^T)^{-1} \tilde{X}=
\frac{1}{\tilde{r}} (ZZ^T+\bar{Z})^{-1} (Z+ \delta Z)
\end{equation*}
Now notice that the columns of $Z$ represent the vertices of a unit $(d-1)$-dimensional hypercube. Also, the norm of each column of $\delta Z$ is bounded by $\frac{r}{8\tilde{r}} < \frac{1}{4}$. Therefore, $\sigma_{\min} (ZZ^T + \bar{Z}) \geq C$, for some constant $C = C(d)$. Hence, for every $1 \leq i \leq q$
\begin{equation*}
 \|(\tilde{X} \tilde{X}^T)^{-1} \tilde{X} e_i\| \leq \frac{1}{\tilde{r}} \sigma_{\min}^{-1}
 (ZZ^T+ \bar{Z})\; \|(Z + \delta Z)e_i\| \leq \frac{C}{r},
\end{equation*}
for some constant $C$. Therefore, all entries of $(\tilde{X} \tilde{X}^T)^{-1} \tilde{X}$ are bounded by $C/r$.
\section{Proof of Remark~\ref{rem:chord_dist}}
\label{App:chord_dist}
Let $\theta$ be the angle between $a$ and $b$ and define $a_{\perp} = \frac{b - \cos(\theta)a}{\|b - \cos(\theta)a\|}$. Therefore, $b = \cos(\theta)a + \sin(\theta) a_{\perp}$. In the basis $(a, a_{\perp})$, we have
\begin{equation*}
aa^T =
 \begin{bmatrix}
 1 & 0 \\
 0 & 0
 \end{bmatrix}, \quad
 bb^T = 
 \begin{bmatrix}
 \cos^2(\theta)& \sin(\theta) \cos(\theta) \\
 \sin(\theta)\cos(\theta) & \sin^2(\theta)
 \end{bmatrix}.
\end{equation*}
Therefore,
\begin{equation*}
\|aa^T - bb^T\|_2 = \bigg \|
\begin{bmatrix}
 \sin^2(\theta)& -\sin(\theta) \cos(\theta) \\
 -\sin(\theta)\cos(\theta) & -\sin^2(\theta)
 \end{bmatrix}
\bigg \|_2 = |\sin(\theta)| = \sqrt{1- (a^Tb)^2}.
\end{equation*}

\section{Proof of Remark~\ref{rem:Weyl}}
\label{App:Weyl}

\begin{proof}
Let $\{\tilde{\lambda}_i\}$ be the eigenvalues of $\tilde{A}$ such that $\tilde{\lambda}_1 \geq \tilde{\lambda}_2 \geq \cdots \geq \tilde{\lambda}_p$. Notice that
\begin{align*}
\|A- \tilde{A}\|_2 &\geq v^T (\tilde{A} -A) v\\
&\geq \tilde{\lambda}_p (v^T \tilde{v})^2 +
\tilde{\lambda}_{p-1} \|P_{\tilde{v}^{\perp}} (v)\|^2 - \lambda_p\\
&= \tilde{\lambda}_p (v^T \tilde{v})^2 +
\tilde{\lambda}_{p-1} (1- (v^T \tilde{v})^2) - \lambda_p\\
&= (\tilde{\lambda}_p - \tilde{\lambda}_{p-1}) (v^T \tilde{v})^2+
\tilde{\lambda}_{p-1} - \lambda_p.
\end{align*}
Therefore, 
\begin{equation*}
(v^T \tilde{v})^2 \geq \frac{\tilde{\lambda}_{p-1} - \lambda_p - \|A -\tilde{A}\|_2}{\tilde{\lambda}_{p-1} - \tilde{\lambda}_p}.
\end{equation*}
Furthermore, due to Weyl's inequality, $|\tilde{\lambda}_i - \lambda_i| \leq \|A - \tilde{A}\|_2$. Therefore,

\begin{equation}
\label{eq:Weyl_bound_prod}
(v^T \tilde{v})^2 \geq \frac{\lambda_{p-1} - \lambda_p - 2 \|A -\tilde{A}\|_2}{\lambda_{p-1} - \lambda_p + 2 \|A - \tilde{A}\|_2},
\end{equation}
which implies the thesis after some algebraic manipulations. 
\end{proof}

%
\newpage
\section{Table of Symbols}
\label{App:tab_notation}
\begin{table}[h!]
\begin{center}
{
\renewcommand{\arraystretch}{1.04}
\begin{tabular}{|c|l|}
\hline
$n$ & number of nodes\\
$d$ & dimension (the nodes are scattered in $[-0.5,0.5]^d$)\\
$L \in \reals^{n \times n}$ & $I - \frac{1}{n}uu^T$, where $I$ is the identity matrix and $u$ is the all-ones vector\\                                                                    
$x_i \in \reals^d$ & coordinate of node $i$,  for $1\le i\le n$\\                                                                                 
$x^{(\ell)} \in \reals^n$ & the vector containing the $\ell^{th}$ coordinate of the nodes, for $1 \le \ell \le d$\\ 
$X \in \reals^{n \times d}$ & the (original) position matrix\\
$\hX$ & estimated position matrix \\
$Q \in \reals^{n \times n}$ & Solution of SDP in the first step of the algorithm\\
$Q_0 \in \reals^{n \times n}$ & Gram matrix of the node (original) positions, namely $Q_0 = XX^T$\\
Subspace $V$ & the subspace spanned by vectors $x^{(1)}, \cdots,x^{(d)}, u$\\
$R \in \reals^{n \times n}$ & $Q- Q_0$\\
$\tilde{R} \in \reals^{n \times n}$ & $P_VRP_V + P_VRP_V^{\perp}+ P_V^{\perp} R P_V$\\ 
$R^{\perp} \in \reals^{n \times n}$ & $P_V^{\perp} R P_V^{\perp}$\\
$\C_i$ & $\{j \in V(G): d_{ij} \le r/2\}$ (the nodes in $\C_i$ form a clique in $G(n,r)$)\\ 
$S_i$ & $\{\C_i\} \cup \{\C_i \backslash k\}_{k \in \C_i}$\\
$\cliq(G)$ & $S_1 \cup \cdots \cup S_n$\\
$\tilde{\C}_i$ & $\{j \in V(G): d_{ij} \le r/2(1/2+1/100)\}$\\
$\tilde{S}_i$ & $\{\C_i, \C_i \backslash i_1,\cdots,\C_i \backslash i_m\}$, where $i_1,\cdots,i_m$ are the $m$ nearest neighbors of node $i$\\
$\cliq^*(G)$ & $\{\tilde{S}_1 \cup \cdots \cup \tilde{S}_n\}$\\
$G$ & $G(n,r)$\\
$\tilde{G}$ & $G(n,r/2)$\\
$G_{ij}$ & the chain between nodes $i$ and $j$\\
$G^*$ & the graph corresponding to $\cliq^*(G)$ (see page 13)\\
$N$ & number of vertices in $G^*$\\
$\L \in \reals^{n \times n}$ & the Laplacian matrix of the graph $G$\\
$\tilde{\L} \in \reals^{n \times n}$ & the Laplacian matrix of the graph $\tilde{G}$\\
$\Omega \in \reals^{n \times n}$ & stress matrix\\
$R_G(X) \in \reals^{|E| \times dn}$ & rigidity matrix of the framework $G_X$\\
$\cR_{G,X}(Y): \reals^{n \times n} \to \reals^{E}$ & For a matrix $Y \in \reals^{n \times d}$, with rows $y_i^T, i=1,\cdots,n$,\\
& $\cR_{G,X}(Y) = R_G(X)\mathcal{Y}$, where $\mathcal{Y} = [y_1^T,\cdots,y_n^T]^T$\\
$x^{(\ell)}_{\Q_i} \in \reals^{|\Q_i|}$ & restriction of vector $x^{(\ell)}$ to indices in $\Q_i$, for $1\le \ell \le d$ and $\Q_i \in \cliq(G)$\\
 $\txl_{\Q_i} \in \reals^{|\Q_i|}$ & component of $x^{(\ell)}_{\Q_i}$ orthogonal to the all-ones vector $u_{\Q_i}$, i.e., $P_{u_{\Q_i}}^{\perp} x^{(\ell)}_{\Q_i}$\\
 $\beta^{(\ell)}_i$ & coefficients in local decomposition of an arbitrary (fixed) vector $v \in V^{\perp}$\\
 & $(v_{\Q_i} = \sum_{\ell=1}^d \beta_i^{(\ell)} \txl_{\Q_i} + \gamma_i u_{\Q_i} + w^{(i)})$\\
  $\beta^{(\ell)} \in \reals^N$ & $(\beta^{(\ell)}_1,\cdots,\beta^{(\ell)}_N)$, for $\ell =1, \cdots d$\\
 $\bar{\beta}^{(\ell)}$ & average of numbers $\beta^{(\ell)}_i$, i.e., ($1/N) \sum_{i=1}^N \beta^{(\ell)}_i$\\
 $\hat{\beta}^{(\ell)}_i$ & $\beta^{(\ell)}_i - \bar{\beta}^{(\ell)}$\\
 $\hat{\beta}_i \in \reals^d$ & $(\hat{\beta}^{(1)}_i, \cdots, \hat{\beta}^{(d)}_i)$, for $i=1,\cdots,N$\\
\hline
\end{tabular}
}
\end{center}
\vspace{-.6cm}
\caption{Table of Symbols}
\label{tab:notation}
\end{table}

\newpage
\bibliographystyle{abbrv}
\bibliography{sigproc}
\end{document}